\documentclass[a4paper,12pt]{article}

      \usepackage[english]{babel}

      \usepackage[latin1]{inputenc}
      \usepackage{graphicx}
      \usepackage{multicol}
      \usepackage{graphics,epsfig}
      \usepackage{amsmath,amsfonts,amssymb,amsthm}
     \usepackage{dsfont}
\usepackage{geometry}
\usepackage[dvipsnames]{xcolor}

 \newtheorem{thm}{THEOREM}[section]
       \newtheorem{cor}[thm]{COROLLARY}
\newtheorem{propo}[thm]{PROPOSITION}

      \newtheorem{rmq}{Remark}[section]
      \newtheorem{lem}[thm]{LEMMA}

\newcommand{\mysection}[1]{\section{#1} \setcounter{equation}{0}}

\font\TenEns=msbm10
\font\SevenEns=msbm7
\font\FiveEns=msbm5
\newfam\Ensfam

\textfont\Ensfam=\TenEns
\scriptfont\Ensfam=\SevenEns
\scriptscriptfont\Ensfam=\FiveEns

\def\rn{{\mathbb{R}^n}}

\newcommand{\re}[1]{(\ref{#1})}

\def\eps{\varepsilon}

\title{Proportionality of components, Liouville theorems and a priori estimates for noncooperative elliptic systems}

\author{\normalsize \bf Alexandre MONTARU$^{(1)}$, Boyan SIRAKOV$^{(2)}$, Philippe SOUPLET$^{(1)}$
\\
\small $^{(1)}$Universit\'e Paris 13, Sorbonne Paris Cit\'e,
\small LAGA, CNRS, UMR 7539,\\
\small F-93430, Villetaneuse, FRANCE. \\
\small $^{(2)}$PUC-Rio, Departamento de Matematica,
\small Rua Marquês de São Vicente 225, \\
\small Gávea, Rio de Janeiro - CEP 22451-900, BRASIL \\
\small \textit{montaru@math.univ-paris13.fr,  bsirakov@mat.puc-rio.br, souplet@math.univ-paris13.fr} }

\date{}

\begin{document}
\maketitle
\begin{abstract}
{
We study qualitative properties of positive solutions of
noncooperative, possibly nonvariational, elliptic systems.
We obtain new classification and Liouville type theorems in the whole Euclidean space, as well as in half-spaces,
 and deduce a priori estimates and existence of positive solutions for related Dirichlet problems.
We significantly improve the known results for a large class of systems involving a balance between repulsive and attractive terms. This class contains systems arising in biological models of Lotka-Volterra type, in physical models of Bose-Einstein condensates and in models of chemical reactions.}

\end{abstract}

\mysection{Introduction}
\label{intro}

This paper is concerned with existence, non-existence and qualitative properties of
classical solutions of nonlinear elliptic systems in the form
\begin{equation}
\label{mainsyst}
\left\{\ {\alignedat2
 -\Delta u&=f(x,u,v), \\
  -\Delta v&=g(x,u,v).
 \endalignedat}\right.
 \end{equation}
 In a nutshell, we will consider noncooperative, possibly nonvariational, systems with nonlinearities  which
have power growth in $u,v$, and in which the reaction terms dominate the absorption
ones. We will be interested in nonexistence or more general classification results
in unbounded domains such as $\rn$ or a half-space in $\rn$, as well as  in their
applications to a priori estimates and existence of positive solutions
of Dirichlet problems in bounded domains.

\subsection{A model case}\label{intro11}
We will illustrate our results by applying them to the system
\begin{equation}
\label{model}
\left\{\ {\alignedat2
 -\Delta u  &=u{\hskip 1pt}v^p\left[a(x)v^q-c(x)u^q\right] + \mu(x) u\qquad &\mbox{ in }&\Omega, \\
  -\Delta v &=v{\hskip 1pt}u^p\left[b(x)u^q-d(x)v^q\right] + \nu(x) v\qquad &\mbox{ in }&\Omega,\\
  u&=0,\;v=0\qquad &\mbox{ on }&\partial\Omega\;(\mbox{if }\partial \Omega\not=\emptyset),
 \endalignedat}\right.
 \end{equation}
 where $\Omega\subseteq\rn$,
\begin{equation}
 \label{cond_pq}
p\ge0,\qquad q>0,\qquad q\ge |1-p|,
 \end{equation}
 and the coefficients $a,b,c,d,\mu,\nu$ are H\"older continuous functions in
 $\overline{\Omega}$, with
 \begin{equation}
\label{cond_abcd}
a,b>0\;\mbox{ in }\;\overline{\Omega},\qquad c,d\ge0\;\mbox{ in }\;\overline{\Omega}.
\end{equation}
Observe \re{model} already covers  a large class of systems satisfied by
stationary states of coupled reaction-diffusion equations, or by standing waves of
Schr\"odinger equations in the typical form
 \begin{equation}
\label{stat_states}
\mathbf{U}_t -\Delta \mathbf{U} = \mathcal{A}(x,\mathbf{U}) \mathbf{U}, \qquad\qquad i\mathbf{U}_t -\Delta \mathbf{U} = \mathcal{A}(x,\mathbf{U})\mathbf{U},
 \end{equation}
where $ \mathbf{U} =(u,v)^T$ and $\mathcal{A}$ is a matrix which describes
the replication rate of and the interaction between the quantities $u$ and $v$.
 Let us mention that for $p=0$ and $q=1$ we obtain a Lotka-Volterra system,
while for $p=0$ and $q=2$ we get a system arising in the theory of Bose-Einstein
condensates and nonlinear optics, which has been widely studied in the recent years.
Systems like \re{model} with $p>0$ appear in models of chemical interactions.
A more detailed discussion and references will be given in Section \ref{intro13},
below.

We will almost always assume that the reaction terms in the system dominate the
 absorption terms, in the following sense
\begin{equation}
\label{diff_dominates}
D:=ab-cd\ge 0 \qquad \mbox{ in }\;\Omega.
\end{equation}

The following two theorems effectively illustrate the more general results below.
 Here and in the rest of the article, $\lambda_1(-\Delta, \Omega)$ denotes the first eigenvalue of $-\Delta$ with Dirichlet boundary conditions in~$\Omega$.

\begin{thm}\label{illus1}
Let $\Omega$ be a smooth bounded domain. Assume that \re{cond_pq}--\re{cond_abcd} hold,
\begin{equation}
\label{cond_illus1}
\inf_\Omega D>0,\qquad p+q<\frac{4}{(n-2)_+}
\end{equation}
 and
\begin{equation}
\label{cond_illus2}
\mu,\nu<\lambda_1(-\Delta, \Omega)\;\mbox{ in }\;\overline{\Omega}.
\end{equation}
Then the system \re{model} has a classical solution $(u,v)$ in $\Omega$, such that $u,v>0$ in $\Omega$. All such solutions are uniformly bounded in $L^\infty(\Omega)$.
\end{thm}

\begin{thm}\label{illus2} Assume that \re{cond_pq}--\re{cond_abcd} hold, 
$\mu=\nu=0$, and $a,b,c,d$ are constants. Let $(u,v)$ be a nonnegative classical solution of \re{model}.
\begin{enumerate}
\item If $\Omega=\rn$ and $D\ge0$ then either $u\equiv0$, or $v\equiv0$, or
    $ u\equiv Kv$ for some unique constant $K>0$. 
\item If $\Omega=\rn$, $D>0$ and $\displaystyle p+q<\frac{4}{(n-2)_+}$, then for some nonnegative constant $C\ge 0$\\
 $${ (u,v)\equiv (C,0)\quad \text{ or }\quad (u,v)\equiv (0,C)}.$$
 If $p=0$ then $C=0$.
 \item If $\Omega$ is a half-space of $\rn$ and $u,v\in L^\infty(\Omega)$, then
  $u=v\equiv0.$
\end{enumerate}
  \end{thm}

  \subsection{A quick overview of our goals and methods}\label{intro12}

 As the previous two theorems show, the two main goals we pursue are:
  \smallskip

(a){\hskip 2mm}obtain classification (or non-existence) results for solutions of \eqref{mainsyst} in $\rn$ or in a half-space of $\rn$. Naturally, to that goal we need to assume some homogeneity of \eqref{mainsyst} in~$(u,v)$.

 \smallskip

(b){\hskip 2mm}prove a priori estimates and existence statements for the Dirichlet problem for \eqref{mainsyst} in bounded domains. The "blow-up" method of Gidas and Spruck yields such results for general nonlinearities $f$ and $g$, whose leading terms are the functions for  which the non-existence theorems in (a) are proved.
\smallskip

This scheme is well-known and has been used widely since the pioneering works \cite{GS1, GS2}. As can be expected, the main effort falls on the classification results in (a). It should be stressed that these classification results must not be viewed only as a step to the existence results in (b), but are of considerable importance in themselves.

It appears that for systems in the whole space or in a half-space,
most methods to prove Liouville type theorems under optimal growth assumptions
are based either on moving planes or spheres and Kelvin transform, and hence
require some rather restrictive cooperativity assumptions (cf.~\cite{Reichel-Zou, FigYang, FigSir, Zou});
or on integral identities such as Pohozaev's identity, and hence require some
variational structure
(cf. \cite{SerrinZou, Sir, SoupletAdv, QScmp, SoupletNHM, DancerWeth}).

However, there are large classes of systems appearing in applications, whose
structure is not treatable by these techniques. One of our  basic observations
is that many such systems have an inherent monotonicity structure expressed by
the following hypothesis
\begin{equation}
\label{condition_fg}
\exists\: K>0\::\:\quad [f(x,u,v)-Kg(x,u,v)][u-Kv]\leq 0  \quad \hbox{ for all $(u,v)\in \mathbb{R}^2$ and $x\in \Omega$, }
\end{equation}
which plays a fundamental role in our nonexistence and classification results.

To fix ideas, we will immediately describe a class of systems that appear in applications, satisfy  \eqref{condition_fg}, but do not seem to be manageable by the well-known methods for establishing Liouville type results. Our techniques naturally extend to even more general systems that satisfy the condition \eqref{condition_fg} (observe that verifying \eqref{condition_fg} for any given system is a matter of simple analysis).

Consider the system
\begin{equation}
\label{system_Spqr}
\left\{\quad{\alignedat2
-\Delta u&=u^rv^p[av^q-cu^q] \\
 -\Delta v&=v^ru^p[bu^q-dv^q].
 \endalignedat}\right.
\end{equation}
In our study of \eqref{system_Spqr} we  always assume that the real parameters $a,b,c,d,p,q,r$ satisfy
\begin{equation}\label{Hyp_abcdpqr}
a, b>0,\quad c,d\ge 0, \qquad p,r\ge 0,\quad q>0,\quad q\ge |p-r|.
\end{equation}

\begin{propo}
\label{Prop-uKv}
Assume (\ref{Hyp_abcdpqr}).
\smallskip

(i) Then  the nonlinearities in system \eqref{system_Spqr} satisfy \eqref{condition_fg}. 
\smallskip

{ (ii) Assume moreover that} $ ab \geq cd$. Then the number $K$ is unique. We have
  $K=1$ if and only if $a+d=b+c$ and  $K>1$ if and only
if $a+d>b+c$. In addition, if $ab>cd$  (resp. $ab=cd$),
then $a-cK^q>0$  (resp. $=0$) and $bK^q-d>0$ (resp. $=0$).
\end{propo}

The proof of this proposition is of course elementary (though tedious), and
 will be given in the appendix.
There is no explicit formula for $K$, except in some special cases.
For instance,
when
 $p=0$ and $r=1$, one easily computes that $K=(\frac{a+d}{b+c})^{1/q}$.
We do not know whether the hypothesis $q\geq |p-r|$ in \eqref{Hyp_abcdpqr}
is necessary for our classification results.  However,
we observe that even the simplest systems of the type \eqref{system_Spqr}, with $c=d=0$, may have a different solution set from what we obtain here, 
when \eqref{condition_fg} (and hence $q\geq |p-r|$) is not satisfied, as shown by the result in \cite[Theorem~1.4(iii)]{QS}.
On the other hand, many models of which we are aware satisfy this hypothesis.
\smallskip

Let us now discuss the use of \eqref{condition_fg}.
 The point of this hypothesis is that for any solution $(u,v)$ of
\eqref{mainsyst},   the nonnegative functions $(u-Kv)_+$ and $(Kv-u)_+$ are {\it subharmonic} in~$\Omega$, which allows for applications of various forms of the maximum principle.

In particular, if \eqref{condition_fg} holds, the domain $\Omega$ is bounded and $u=Kv$ on $\partial \Omega$, then the classical maximum principle implies the classification property
\begin{equation}
\label{uKv}
u \equiv Kv\quad\hbox{ in }\Omega,
\end{equation}
which reduces the system to a single elliptic equation.
Our basic goal will be to prove \eqref{uKv} for unbounded domains like the whole
space $\mathbb{R}^n$ or a half-space, where additional work and hypotheses are needed. We comment briefly on these next.

First, when $\Omega$ is a half-space and $u,v$ have sublinear growth at infinity, \eqref{uKv} is a consequence of the Phragmén-Lindel\"of maximum principle (a tool which is not often encountered in the context of Liouville theorems for nonlinear systems). This classifies bounded solutions in a half-space, which in particular is sufficient for the application of the blow-up method. We obtain nonexistence and classification results for general unbounded solutions in a half-space as well - then some supplementary assumptions are unavoidable. The proofs of these more general results use properties of spherical means of functions in a half-space, as well as  a general nonexistence result for weighted elliptic inequalities in cones from \cite{armsir}.

Next, proving \eqref{uKv} in the whole space is where we encounter most difficulties. Probably the most important and novel observation we make is that under our assumptions the functions $Z=\min(u,Kv)$ and $W=|u-Kv|$ satisfy the inequality
\begin{equation}
\label{anticoersyst}
-\Delta Z \ge cW^\mu Z^r \qquad\hbox{in }\;\rn,
\end{equation}
 and, in some cases,
\begin{equation}
\label{coersyst}
\Delta W \ge c Z^p W^{\gamma} \qquad\hbox{in }\;\rn,
\end{equation}
 for appropriate $\mu,\gamma\ge 1$, $c>0$.   It is worth observing that in \eqref{anticoersyst} the {\it superharmonic} function $Z$ satisfies an anti-coercive elliptic inequality with a {\it subharmonic} weight $W^\mu$, while in \eqref{coersyst} the subharmonic function $W$ satisfies a coercive inequality with a weight which is a power of a superharmonic function. We do not know of any other work where such combinations of inequalities and weights appear. By using properties of subharmonic and superharmonic functions and by adapting the methods for proving nonexistence of positive solutions of inequalities from \cite{armsir} (for \eqref{anticoersyst}) and from \cite{Lin2}  (for \eqref{coersyst}), we show that under appropriate restrictions on the exponents $p,r$, we have $W\equiv0$.

\smallskip

The idea of showing nonexistence results by first proving the property (\ref{uKv})
was used earlier in~\cite{Lou, DLS} 
for a particular Lotka-Volterra type system, and more recently in \cite{QS,Fa,Damb}, where
 some partial use of  \eqref{condition_fg} with $K=1$ was also made. To our
 knowledge the present paper is the first systematic study of systems whose
nonlinearities satisfy~\eqref{condition_fg}. Our results very strongly improve
on previous ones, both in the generality of the systems considered, and in the
results obtained, even when applied to  particular systems. Our methods,
in particular the above observations, appear to be new.
\smallskip

Finally, as far as step (b) above is concerned, we recall that  uniform a priori
estimates and existence of positive solutions of Dirichlet problems
associated with asymptotically homogeneous systems in bounded domains can be obtained via the rescaling (or blow-up) method
of Gidas and Spruck \cite{GS1} combined with known topological degree arguments (see for instance~\cite{CMM, Lou, DLS, FigYang, FigSir, Zou}
for systems). Applying this method requires nonexistence theorems  for the limiting "blown-up" system in the whole space
and in the half-space. We will follow the same scheme here; however, as an additional and nontrivial difficulty with respect to the cases treated in~\cite{CMM, Lou, DLS, FigYang, FigSir}, we will need to deal with the fact that
many of the limiting systems that we obtain
 admit semi-trivial solutions in the whole space, of the form $u=0,v=C$ or $u=C,v=0$, with $C>0$
(for instance system~(\ref{system_Spqr}) with $p,r>0$).
Additional arguments are thus needed to rule out the occurrence of such limits (see Remark~\ref{remruleout}).
\smallskip

\subsection{Some systems that appear in applications, to which our results apply}\label{intro13}

The results in Section \ref{intro11} apply in particular  to the following two systems which we already mentioned
$$
(LV)\!\left\{{\alignedat2
-\Delta u&=u\bigl[a(x)v-c(x)u+\mu(x)\bigr]\\
 -\Delta v&=v\bigl[b(x)u-d(x)v+\nu(x)\bigr],
 \endalignedat}\right.
 \qquad
 (BE)\!\left\{{\alignedat2
-\Delta u&=u\bigl[a(x)v^2-c(x)u^2+\mu(x)\bigr]\\
 -\Delta v&=v\bigl[b(x)u^2-d(x)v^2+\nu(x)\bigr].
 \endalignedat}\right.
$$

The first of these two systems is of Lotka-Volterra type, and appears as a model of symbiotic interaction of biological species. In (LV) the logistic terms $(\mu-cu)u$ and $(\nu-dv)v$ take into account the reproduction {and the limitation of resources} within each species, while the $uv$-terms represent the interaction  between the two species.
A positive solution then corresponds to a coexistence state -- see for instance  \cite{KoLe, Lou, DLS} and the references therein for more details on the biological background.

The system (BE) arises in models of Bose-Einstein condensates which involve two different quantum states, as well as in nonlinear optics. In particular, one gets (BE) when looking for standing waves
of an evolutionary cubic  Schr\"odinger system.
In the present case, the interspecies interaction is attractive, while the self-interaction is repulsive or neutral,
leading to phenomena of symbiotic solitons.
We refer to  \cite{Perez-Garcia, Che} for a description of physical phenomena that lead to such systems. These references include  systems with spatially inhomogeneous coefficients.

For (LV) and (BE), we get the following result as a direct consequence of Theorem~\ref{illus1}.

\begin{cor}
\label{AprioriBoundLV}
Assume $\Omega$ is a smooth bounded domain,   $a, b, c, d,\mu,\nu$ are H\"older continuous in {$\overline\Omega$}, 
and \eqref{cond_abcd}, \eqref{cond_illus2} hold. Assume further that $n\le 5$ for (LV),  $n\le 3$ for (BE).

If
\begin{equation}
\label{hypApriori2_0b}
\inf_{x\in \Omega}\,\left[{a(x)b(x)}-{c(x)d(x)}\right]>0,
\end{equation}
then there exists at least one positive classical solution of (LV) or (BE) in $\Omega$, such that $u=v=0$ on $\partial\Omega$. All such solutions are bounded above by a constant which depends only on $\Omega$, and the uniform norms of $a, b, c, d,\mu,\nu$.
\end{cor}

Observe that \eqref{hypApriori2_0b} cannot be removed, as simple examples show. For instance, if $a=b=c=d$ and 
$ \mu=\nu =0$ 
 in (LV) or (BE), by adding up the two equations we see that any nonnegative solution of the Dirichlet problem vanishes identically.

In spite of the huge number of works on Lotka-Volterra systems (giving a reasonably complete bibliography is virtually impossible), this corollary represents an improvement on known results for (LV) --- see \cite[Theorem 7.4]{DLS},
where a more restrictive assumption than (\ref{hypApriori2_0b}) was made on the functions $a, b, c, d$.

For the system (BE), most of the previously known statements on a priori estimates and existence
concerned the case of reversed interactions ($a,b$ positive and $c,d$ negative; or $a,b,c,d$ negative); see \cite{BDW, BPB, DWW, TTVW, QScmp, FigSir}.
The self-repulsive case which we consider here was also studied in \cite{Lin}, where positive solutions are constructed by variational methods under the additional hypothesis that $a=b$ and $a,b,c,d$ are large constants.
Thus Corollary \ref{AprioriBoundLV} completes these works, providing optimal results for the case of attractive interspecies interaction, and repulsive or neutral intraspecies interaction.

\medskip {
Finally, we point out the following third example, which is a special case of a class of systems arising in the modelling of general
 chemical reactions
\begin{equation}
\label{systChem1}
\left\{\quad{\alignedat2
u_t-\Delta u&=uv\bigl[a(x)v-c(x)u\bigr], &\qquad&t>0,\ x\in \Omega,\\
\gamma v_t-\Delta v&=uv\bigl[b(x)u-d(x)v\bigr], &\qquad&t>0,\ x\in \Omega,\\
u&=v=0, &\qquad&t>0,\ x\in \partial\Omega,\\
 \endalignedat}\right.
\end{equation}
where $\Omega$ is a bounded domain and $\gamma>0$.
See for instance equation~(1) in \cite{DF} and equations (3.1), (3.5)  in \cite{Pierre},
 as well as the other examples and references given in these works (note that more general power-like behaviour in the nonlinearities can be considered as well).
Specifically, system (\ref{systChem1}) in the case $a(x)=d(x)$ and $b(x)=c(x)$ describes
the evolution of the concentrations of two chemical {molecules} $A$ and $B$
in the reversible reaction
$$A+2B \ {{k_1\atop \longrightarrow} \atop {\longleftarrow\atop k_2}} \ 2A+B,$$
under inhomogeneous catalysis with reaction speeds $k_1=a(x)$, $k_2=b(x)$,
and absorption on the boundary.
(Note that the net result of the reaction is $ B \ {{{}\atop \longrightarrow} \atop {\longleftarrow\atop {}}}  A$
and that the molecules $A, B$ should thus be isomeric.)
In this case, it is easy to see by considering $u+v$ that
the only nonnegative equilibrium is $(u,v)=(0,0)$. Hence,  Theorem~\ref{illus1} shows the existence of a bifurcation phenomenon for the
stationary system associated with (\ref{systChem1}), precisely at $a(x)=d(x)$ and $b(x)=c(x)$. Indeed, assume $n\le 3$ and let the H\"older continuous functions $a,b,c,d$ satisfy $ab= cd+\varepsilon$, $a,b>0$ and $c,d\ge0$ in~$\overline{\Omega}$.
Then there exists a positive steady state beside the trivial one, for each $\eps>0$.

It is worth noticing that the discussion in  \cite{Pierre} (see eqn.~(3.6) in that paper) provides a physical explanation
as to why the case $ab> cd$ differs strongly from $ab\le cd$.
As is pointed out in \cite{Pierre},  in the case of constant coefficients, $ab\le cd$ guarantees that the
system \eqref{systChem1} exhibits control of mass (we refer to \cite{Pierre}
for definitions), or, in other words,  the absorption in the system controls the
reaction. Under this assumption, it can be shown that any global and bounded solution converges uniformly to $(0,0)$ as $t\to\infty$
(however, whether or not some solutions may blow up in finite time is a highly nontrivial question in general -- see \cite{Pierre} and the references therein).
It should then come as no surprise that the case $ab>cd$, in which no  control of mass is available, is delicate to study, even in the stationary (elliptic) case.}

\mysection{Main results}\label{mainres}

We will only consider classical solutions, for simplicity. Observe that under our hypotheses on $f$, $g$, any continuous weak-Sobolev solution of \eqref{mainsyst} is actually classical, by standard elliptic regularity.

In what follows, we say that $(u,v)$ is semi-trivial  if $u\equiv 0$ or $v\equiv 0$.
We say that $(u,v)$ is positive if $u,v>0$ in the domain where a given system is set.

\subsection{Classification results in the whole space}\label{mainres21}

In this section we study the system \eqref{system_Spqr} in $\rn$.
The following theorem  plays a pivotal role in the paper and is probably its  most original result.

\begin{thm}
\label{thm_Spqr}
Assume (\ref{Hyp_abcdpqr}) holds and $ab \geq cd $. Let $K>0$ be the constant from Proposition \ref{Prop-uKv} and $(u,v)$ be a positive classical solution of \eqref{system_Spqr} in $\rn$.

(i) Assume that
\begin{equation}\label{HypthmSpqr}
r\le \frac{n\phantom{_+}}{(n-2)_+}.
\end{equation}
If $p+q<1$, assume in addition that $(u,v)$ is bounded. Then $u\equiv Kv$.

(ii) Assume that
\begin{equation}\label{HypthmSpqr2}
p\le \frac{2\phantom{_+}}{(n-2)_+} \quad\hbox{and}\quad c, d>0.
\end{equation}
If $q+r\le 1$, assume in addition that $(u,v)$ is bounded. Then $u\equiv Kv$.
\end{thm}

We stress that, remarkably, Theorem \ref{thm_Spqr} includes {\it critical and supercritical} cases, since no
upper bound is imposed on the {\it total degree} $\sigma:=p+q+r$ of the system \eqref{system_Spqr}.

Theorem \ref{thm_Spqr} provides  a classification of positive solutions
of (\ref{system_Spqr}) in $\rn$. Specifically, the set of positive solutions of (\ref{system_Spqr}) is given by $(u,v)=(KV,V)$, where $V$ is either a  positive harmonic function, hence constant (if $ab=cd$) or $V$ is a solution of
\begin{equation}\label{gseq}-\Delta V=c_1 V^\sigma\qquad \mbox{in }\; \mathbb{R}^n,
\end{equation}
with $c_1=K^p (bK^q-d)>0$ (if $ab>cd$, by Proposition~\ref{Prop-uKv}).
It is well known that positive solutions of \eqref{gseq} exist precisely if $n\ge 3$ and $\sigma\ge (n+2)/(n-2)$.
They are moreover unique up to rescaling and translation, and explicit,  if $\sigma=(n+2)/(n-2)$ (see~\cite{CGS}).
For some related classification results for cooperative systems with $c=d=0$ in the critical case, which use the method of moving planes, see \cite{GuoLiu, LiMa}.
\smallskip

Theorem \ref{thm_Spqr} significantly improves  the results from~\cite{QS} concerning system (\ref{system_Spqr}) (see \cite[Theorem~2.3]{QS}).
There, only the case $a=b$, $c=d$ (hence $K=1$) was considered and, for that case, much stronger restrictions than (\ref{HypthmSpqr}) or (\ref{HypthmSpqr2}) were imposed.
\smallskip

Combining Theorem \ref{thm_Spqr} with known results on scalar equations yields the following striking Liouville type result for the noncooperative system (\ref{system_Spqr}), with an optimal growth assumption on the nonlinearities.

\begin{thm}
\label{thm_SpqrLiouv}
Assume (\ref{Hyp_abcdpqr}), $ab>cd$, and
$$ \sigma:=p+q+r<\frac{n+2\phantom{_+}}{(n-2)_+}.$$
\smallskip

(i) Then system (\ref{system_Spqr}) does not admit any positive, classical, bounded solution.
\smallskip

(ii)  Assume in addition $$p+q\ge 1, \qquad\hbox{or}\qquad p\le \frac{2}{(n-2)_+},\qquad\hbox{or}\qquad \sigma\le \frac{n\phantom{_+}}{(n-2)_+}
 $$
(note this hypothesis is  satisfied  in each one of the "physical cases" $q\ge1$ or $n\le 4$).
 Then system (\ref{system_Spqr}) does not admit any positive classical (bounded or unbounded) solution.
\end{thm}

Once positive solutions are ruled out, it is natural to ask about nontrivial nonnegative solutions
(and this will be important in view of our applications to a priori estimates, below). The following result is a simple consequence of Theorem \ref{thm_SpqrLiouv} and the strong maximum principle.

\begin{cor}
\label{rmqNonneg}
Under the hypotheses of Theorem~\ref{thm_SpqrLiouv}(i) (resp., \ref{thm_SpqrLiouv}(ii)),
assuming in addition $q+r\ge 1$,
any nonnegative bounded (resp., nonnegative) solution of (\ref{system_Spqr}) is in the form
$(C_1,0)$ or $(0,C_2)$, where $C_1, C_2$ are nonnegative constants.

Moreover, if in addition $p=0$, $r>0$ and $c>0$ (resp., $d>0$), then $C_1=0$ (resp., $C_2=0$),
whereas, if $r=0$, then $C_1=C_2=0$.
\end{cor}

We end this subsection with several remarks on the hypotheses in the above theorems.
It is not known whether or not the restrictions (\ref{HypthmSpqr}), (\ref{HypthmSpqr2}) are optimal  for the property $u\equiv Kv$. However, the following result shows that
this property may fail if $p$ and $r$ are large enough.

\begin{thm}
\label{thm_SpqrCounter}
Let $n\ge 3$ and consider system (\ref{system_Spqr}) with $p=r>(n+2)/(n-2)$, $q>0$ and $a=b=c=d=1$.
Then there exists a positive solution such that $u/v$ is not constant.
\end{thm}

We remark that if $c=0$ or $d=0$ then we can show that
at least one of the components dominates the other, without restrictions on $p$ or $r$.

\begin{propo}
\label{propo_cdzero}
Assume (\ref{Hyp_abcdpqr}) and $c=0$ or $d=0$.
 Assume that either $(u,v)$ is  bounded or $\max(p+q,q+r)>1$. 
Then either $u\ge Kv$ in $\mathbb{R}^n$ or $u\le Kv$ in $\mathbb{R}^n$.
\end{propo}

\begin{rmq}
\label{rmqRad}
 If $u$ and $v$ are assumed to be  radially symmetric, it is easy to show that we  have $u\ge Kv$ or $u\le Kv$, only under the assumptions (\ref{Hyp_abcdpqr}) and $ ab \ge cd$ {(see the end of section~4.1).}
\end{rmq}

Next, we recall that the property $u=Kv$ is known to be true  for all nonnegative solutions of (\ref{system_Spqr}) provided $p=0$ (so that the system is cooperative), and $q\ge r>0$, $c, d>0$, $q+r>1$.
The proof of this fact (see \cite{Lou, Damb, Fa})
relies on the observation that the function $w=(u-Kv)_+$ satisfies
$\Delta w\ge c_1(u^{q+r-1}+v^{q+r-1})w$ for some constant $c_1>0$, which leads to the ``coercive'' elliptic inequality
\begin{equation}
\label{IneqCoerc}
\Delta w\ge c_1w^{q+r}\quad \mbox{ in }\;\mathbb{R}^n.
\end{equation}
It then follows from a classical result of Keller and Osserman (see also Brezis~\cite{Br}) that $w\equiv 0$, hence $u\equiv Kv$ (after exchanging the roles of $u,v$).
The same idea applies in the half-space, under homogeneous Dirichlet boundary conditions.
However, this argument fails if $p>0$, or if $c$ or $d=0$,
since one does not obtain a coercive equation like~(\ref{IneqCoerc}).
Nevertheless, we will be able to use some more general coercivity properties in the proof of Theorem \ref{thm_Spqr} for $p\le 2/(n-2)$, see \eqref{coersyst}.

 Finally, we recall that the case $ab<cd$ in system (\ref{system_Spqr}) is very different, since the absorption features then become dominant.
For instance, if $p=0$ and $ab<cd$, then any nonnegative solution of (\ref{system_Spqr}) has to be trivial if  $q+r>1$,
in sharp contrast with the case $ab>cd$ (when nontrivial solutions $(u,Ku)$ exist if $q+r\ge (n+2)/(n-2)$).
Indeed, by Young's inequality, one easily checks that $w=u+tv$ satisfies
\eqref{IneqCoerc} for suitable $t, c_1>0$, hence $w\equiv 0$. An interesting question, though outside the scope of this paper, is  to determine the optimal conditions on $p, q, r\ge 0$ under which classification results can be proved, when $ab<cd$.

\subsection{Classification results  in the half-space}\label{mainres22}

We begin with a rather general classification result for system (\ref{mainsyst}) on the half-space $\mathbb{R}^n_+=\{x\in \rn\::\: x_n>0\}$,
under the basic structure assumption~(\ref{condition_fg}).

A function
$u\, :\, \mathbb{R}^n_+\rightarrow \mathbb{R}$ is said to have \textit{sublinear growth} if
$u(x)=o(|x|)$ as $|x|\to\infty$, $x\in \mathbb{R}^n_+$.
The following theorem  classifies positive solutions with sublinear growth in $\mathbb{R}^n_+$, and implies nonexistence results  by reducing the system to a scalar equation.
It will thus be sufficient,
along with Liouville type results for bounded solutions in the whole space (stated in Section \ref{mainres21}),
in order to prove a priori estimates via the blow-up method.

\begin{thm}
\label{thm_general}
Assume that (\ref{condition_fg}) holds.
Let $(u,v)$ be a classical solution of~(\ref{mainsyst}) in $\mathbb{R}^n_+$,
such that $u=Kv$ on $\partial \mathbb{R}^n_+$.
If $u$ and $v$ have sublinear growth, then $$u\equiv Kv\quad\mbox{ in }\; \mathbb{R}^n_+.$$
\end{thm}

\noindent This theorem is a consequence of the Phragmèn-Lindel\"of
maximum principle.

\begin{rmq}
\label{rmq1.1}
Observe that we did not make any assumption on the sign  or on the growth of the nonlinearities $f$ and $g$. Therefore, supercritical nonlinearities can be allowed.
\end{rmq}

Theorem \ref{thm_general} can be used to deduce Liouville type theorems for noncooperative systems. We have for instance the following result, which applies to the system (\ref{system_Spqr}).

\begin{cor}
\label{thm_half2}
Assume that (\ref{condition_fg}) holds  for some $K>0$, and there exist constants $c>0$ and $p>1$ such that
$$f(x,Ks,s)=cs^p,\quad s\ge 0.$$
Then system (\ref{mainsyst}) has no nontrivial, bounded, classical nonnegative solution in $\mathbb{R}^n_+$, such that $u=v=0$ on the boundary $\partial \mathbb{R}^n_+$.
\end{cor}

This corollary is obtained  by combining Theorem \ref{thm_general} with a recent result \cite{CLZ}, which
guarantees that, for any $p>1$, the scalar equation $-\Delta u=u^p$ has no positive, bounded, classical solution in the half-space, which vanishes on the boundary (this was known before under additional restriction on $p$, see~\cite{GS1, Dan, Fa1}).
\smallskip

Under further assumptions on the nonlinearities, namely positivity
(one may think of $c=d=0$ in \eqref{system_Spqr}), we obtain classification results in the half-space, without making growth restrictions on the solutions.

\begin{thm}
\label{thm_systeme_S_rsd_generalise}
Let $p,q,r,s\geq 0$.
We assume that $f,g$ satisfy condition (\ref{condition_fg}) for some constant $K>0$ and that, for some $c>0$,
\begin{equation}
\label{condition_fg2}
f(x,u,v)\geq c \;u^r \;v^{p} \text{\quad  and \quad} g(x,u,v)\geq c\; u^{q}\; v^s \quad \text{for all $u,v\geq 0$ and $x\in \mathbb{R}^n_+$}.
\end{equation}
Let $(u,v)$ be a nonnegative classical solution of (\ref{mainsyst}) in $\mathbb{R}^n_+$,
such that $u=Kv$ on $\partial \mathbb{R}^n_+$.
\begin{itemize}
\item [(i)] Either $u\leq Kv$ or $u\geq Kv$ in $\mathbb{R}^n_+$.
\item [(ii)] If
\begin{equation}
\label{condition_rnsd}
 r \leq \frac{n+1+p}{n-1}\quad\hbox{or}\quad q \leq \frac{1+s}{n-1},
\end{equation}
and
\begin{equation}
\label{condition_rnsd2}
s \leq\frac{n+1+q}{n-1} \quad\hbox{or}\quad p \leq \frac{1+r}{n-1},
\end{equation}
then either
$u\equiv Kv $ or $ (u,v) $  is semitrivial.
\item[(iii)] If \eqref{condition_rnsd}-\eqref{condition_rnsd2} hold and  $\min(p+r,q+s)\leq (n+1)/(n-1)$, then $(u,v)$ is semitrivial.
\end{itemize}
\end{thm}

Theorem~\ref{thm_systeme_S_rsd_generalise} complements~\cite[Theorem~1.2]{QS},
which concerned similar problems in  $\mathbb{R}^n$.

\begin{rmq}
\label{rmqOptim}
 The restrictions
(\ref{condition_rnsd})--(\ref{condition_rnsd2}) are unlikely to be optimal, since they are strongly related to nonexistence results for inequalities in the half-space. Recall that $-\Delta v = v^p$ has no positive solutions vanishing on the boundary for each $p>1$, while the same is valid for  $-\Delta v \ge v^p$ if and only if $p\leq (n+1)/(n-1)$.
\end{rmq}

The proof of Theorem~\ref{thm_systeme_S_rsd_generalise} makes use of a generalization of Theorem \ref{thm_general}, which we state next. If $w$ is a continuous function in $\overline{\mathbb{R}^n_+}$, we denote with $[w]$ its half-spherical mean, defined by
$$
[w](R)=\frac{1}{ |S_R^+|}\int_{S_R^+} \frac{w(x)}{R}\;\frac{x_n}{R}\,d\sigma_R(x),
$$
for each $R>0$, where $S_R^+=\{x\in \mathbb{R}^n_+,\, |x|=R \}$.

\begin{thm}
\label{thm_general2}
Assume that (\ref{condition_fg}) holds.
Let $(u,v)$ be a classical solution of~(\ref{mainsyst}) in $\mathbb{R}^n_+$,
such that $u=Kv$ on $\partial \mathbb{R}^n_+$.
If
\begin{equation}
\label{condition_thm_general2}
\underset{R \rightarrow \infty }{\lim\inf} [(u-Kv)_+](R)=0 \text{\qquad and \qquad }
\underset{R \rightarrow \infty }{\lim\inf} [(Kv-u)_+](R)=0,
\end{equation}
then $u\equiv Kv.$
\end{thm}

\begin{rmq}\label{rmqphrlind}  If $u,v$ have sublinear growth then $\underset{R \rightarrow \infty }{\lim} [|u|](R)=\underset{R \rightarrow \infty }{\lim} [|v|](R)=0$, which in turn implies (\ref{condition_thm_general2}).  Hence Theorem \ref{thm_general} is a consequence of Theorem \ref{thm_general2}.
\end{rmq}

\subsection{A priori estimates and existence in bounded domains}\label{mainres23}

We consider the Dirichlet problem
\begin{equation}
\label{system_pqr_Dir}
\left\{\quad{\alignedat2
-\Delta u&=u^rv^p\bigl[a(x)v^q-c(x)u^q\bigr]+\mu(x) u, &\qquad&x\in \Omega,\\
 -\Delta v&=v^ru^p\bigl[b(x)u^q-d(x)v^q\bigr]+\nu(x) v, &\qquad&x\in \Omega,\\
u&=v=0, &\qquad&x\in \partial\Omega,\\
 \endalignedat}\right.
\end{equation}
where $\Omega$ is a smooth bounded domain of $\mathbb{R}^n$.
For simplicity,  here we restrict ourselves to linear lower order terms.
Further results,  for systems with more general lower order terms, will be given in Section~6.
Note that, due to the space dependence of the coefficients $a, b, c, d$ and to the presence of the lower order terms,
the right-hand side of system (\ref{system_pqr_Dir}) does not satisfy~(\ref{condition_fg}), in general.
Therefore, system (\ref{system_pqr_Dir}) cannot be directly reduced to a scalar problem via the property $u\equiv Kv$.

\begin{thm}
\label{AprioriBound}
Let $p, r\ge 0$, $q>0$, and
\begin{equation}
\label{hypApriori1_0}
q\ge |p-r|,\qquad q+r\ge 1, \qquad r\le 1,\qquad 1<p+q+r<\frac{n+2\phantom{_+}}{(n-2)_+}.
\end{equation}
Let  $a, b, c, d,\mu,\nu\in C(\overline\Omega)$ satisfy $a, b>0$, $c, d\ge 0$ in
$\overline\Omega$ and
\begin{equation}
\label{hypApriori2_0}
\inf_{x\in \Omega}\,\left[a(x)b(x)-c(x)d(x)\right]\,>0.
\end{equation}

(i) Then there exists $M>0$, depending only on $p,q,r$, $\Omega$, and the uniform norms of $a, b, c, d,\mu,\nu$, such that any positive classical solution $(u,v)$ of
(\ref{system_pqr_Dir}) satisfies
$$ \sup_{\Omega} u\leq M ,\quad \sup_\Omega v\leq M.$$

(ii) Assume in addition that $a, b, c, d,\mu,\nu$ are H\"older continuous
and that $\mu,\nu<\lambda_1(-\Delta, \Omega)\;\mbox{ in }\;\overline{\Omega}$.
Then there exists at least one positive classical solution of (\ref{system_pqr_Dir}).
\end{thm}

As we already observed, Theorem \ref{AprioriBound} seems to be new even for very particular cases of \eqref{system_pqr_Dir}, for instance the system (BE) from Section \ref{intro13}.

The rest of the paper is organized as follows.
In the preliminary Section 3 we state some essentially known nonexistence results for scalar inequalities with weights. In Section 4 we prove the main classification and Liouville type results for the
repulsive-attractive system (\ref{system_Spqr})
in the whole space.
In Section 5 we introduce the half-spherical means, establish their monotonicity properties and
prove Theorems \ref{thm_general} and \ref{thm_general2}. Then we prove
 some further properties of half-spherical means of
superharmonic functions,   and deduce Theorem \ref{thm_systeme_S_rsd_generalise}.
Finally, Section 6 is devoted to a priori estimates by the rescaling method
and existence by topological degree arguments. In the appendix we gather some elementary computations related to Proposition \ref{Prop-uKv}.

\mysection{Preliminary results. Liouville theorems for\\ weighted inequalities in unbounded domains. }

 In this section we state three essentially known  nonexistence results for scalar elliptic inequalities.
We require such properties  both for  inequalities with source and for inequalities with  absorption.

In the rest of the paper, a {\it weak} solution of an (in)equality in a given domain $\Omega\subset \rn $ will mean a function in $H^1_{{ loc}}(\Omega)\cap C(\overline{\Omega})$ which verifies the given (in)equality in the sense of distributions.

We begin with the following Liouville type result for  weighted
elliptic inequalities with space dependence in an exterior domain of the half-space.

\smallskip

\begin{lem}
\label{lem_estimation_integrale_Liouville}
Let $r\ge 0$ and $u$ be a nonnegative weak solution of :
\begin{equation}
\label{ellipt_ineq}
-\Delta u \geq h(x)\,u^r \text{\qquad on }\; \mathbb{R}^n_+\setminus B_1,
\end{equation}
where   $h\geq0$  on $\mathbb{R}^n_+\setminus B_1$ and there exists  $\kappa>-2$ such that $\kappa+r\geq -1$, and
$$
h(x)\ge c |x|^{\kappa}\qquad\mbox{in the cone }\; \{ x\::\: x_n\geq\delta |x|\}\setminus B_1,
$$
for some constants $c,\delta>0$.

If $$0\leq r\leq \frac{n+1+{\kappa}}{n-1},$$ then $u=0$.
\end{lem}

\begin{proof}This follows from Theorem 5.1 or Corollary 5.6  in \cite{armsir}. Note that theorem was stated for $h(x) = c |x|^\kappa$ but its proof contains the statement of Lemma \ref{lem_estimation_integrale_Liouville}. As is explained in Section 3 of \cite{armsir}, the results in that paper hold for any notion of weak solution, for which the maximum principle and some related properties are valid.
\end{proof}

\begin{rmq} We will apply Lemma \ref{lem_estimation_integrale_Liouville}
 with  $h$ in the form $h(x) = cx_n^s|x|^{-m}$.
\end{rmq}

The next result plays a crucial role in our proofs below.

\begin{lem}
\label{eigenvalue_V}
Assume $0\le r\le n/(n-2)_+$ and let $V\in C(\mathbb{R}^n)$, $V\ge 0$, $V\not\equiv0$ be such that
\begin{equation}
\label{hypVBR}
\liminf_{R\to\infty} R^{-n}\int_{B_R\setminus B_{R/2}} V(x)\, dx>0.
\end{equation}
Let  $z\ge 0$ be a weak solution of
\begin{equation}
\label{eqVBR}
-\Delta z\ge V(x)z^r\qquad \mbox{in }\;\mathbb{R}^n.
\end{equation}
Then $z\equiv 0$.
\end{lem}

The point is that in inequality (\ref{eqVBR}) the potential $V(x)$ is not assumed to be bounded below by a positive constant (in which case the result is well-known - see for instance \cite{MiPo}), but only in {\it average on large annuli}.

In particular, Lemma \ref{eigenvalue_V} applies if $V\gneqq 0$ is a {\it subharmonic} function. Indeed, the mean-value inequality and the well-known fact that for each subharmonic function $V$ the spherical average $\bar V(R) =  \oint_{\partial B_R} V $ is nondecreasing in $R$ easily imply that (here $\oint$ stands for the  average integral)
$$
\oint_{B_R} V(x)\, dx \le \frac{n}{R} \int_0^R \bar V(r) \,dr \leq \frac{2n}{R} \int_{R/2}^R  \bar V(r) \,dr  \le C(n) \oint_{B_R\setminus B_{R/2}} V(x)\, dx,
$$
hence, for each $x_0\in \rn$
$$
C(n)\liminf_{R\to\infty} \oint_{B_R\setminus B_{R/2}} V(x)\, dx \geq \liminf_{R\to\infty}  \oint_{B_R} V(x)\, dx =  \liminf_{R\to\infty}  \oint_{B_R(x_0)} V(x)\, dx \geq V(x_0),
$$
which implies, for each subharmonic $V\gneqq0$ and some positive constant $c(n)$,
$$
\liminf_{R\to\infty} R^{-n}\int_{B_R\setminus B_{R/2}} V(x)\, dx \ge c(n)\, \sup_{\rn} V.
$$
\smallskip

 Lemma \ref{eigenvalue_V} can be proved through a slight modification of the argument introduced in \cite{armsir}. We will give a full and simplified proof, for completeness.

We first recall the following "quantitative strong maximum principle".
\begin{lem}\label{lem_qsmp} Let $\Omega$ be a smooth bounded domain and $K$ be a compact subset of $\Omega$. There exists a constant $c>0$ depending only on $n$, $K$, dist$(K,\partial \Omega)$, such that if $h$ is a nonnegative bounded function and $ u$ satisfies the inequality
$$
-\Delta u \geq h \quad \mbox{in }\; \Omega,
\qquad
\mbox{then}\qquad
\inf_{K} u \geq c \int_{K} h(x)\,dx.
$$
\end{lem}

For a simple proof of  Lemma \ref{lem_qsmp} one may consult Lemma 3.2 in \cite{brecab}. Lemma \ref{lem_qsmp} can also be seen as a consequence of the fact that the Green function of the Laplacian in any domain is strictly positive away from the boundary of the domain.

\smallskip

\begin{proof}[\underline{Proof of Lemma \ref{eigenvalue_V}}] If $n\leq 2$ Lemma \ref{eigenvalue_V} is immediate, since every positive superharmonic function in $\mathbb{R}^2$ is constant.

Suppose now $n\geq 3$ and $u$ is a solution of \eqref{eqVBR}. Set $u_R(x) := u(Rx)$ and  $m(R) :=\inf_{\partial B_R}u = \inf_{\partial B_1} u_R$. By the maximum principle $m(R)=\inf_{ B_R}u = \inf_{B_1} u_R$ and $m(R)$ is nonincreasing in $R$.

Observe that \eqref{hypVBR} is equivalent to the existence of $R_0>0$ and $c_0>0$ such that $$\int_{B_1\setminus B_{1/2}} V(Rx)\,dx \geq c_0>0\quad\mbox{ for }R\ge R_0.$$ From now on we assume that $R\ge R_0$. Since $u_R$ is a solution in $\rn$ of the inequality
$$
-\Delta u_R \geq R^2 V(Rx) u_R^p,$$
 we can apply Lemma \ref{lem_qsmp} with $\Omega = B_2$ and $K=\bar B_1$  and deduce  $$
m(R) \geq cR^2 m(R)^p,
$$
for some $c>0$. If $p\leq 1$ this is a contradiction, since $m(R)$ is nonincreasing in $R$.
If $p>1$ we get
\begin{equation}\label{eqi1}
m(R) \leq C R^{-\frac{2}{p-1}}.
\end{equation}

Since $u$ is superharmonic, the maximum principle implies that
$$
u(x) \geq m(1) |x|^{2-n}\quad \mbox{ in } \; \rn\setminus B_1 ,
$$
and hence
\begin{equation}\label{eqi2}
m(R) \geq c R^{2-n}\quad \mbox{ for  } \; R\geq 1.
\end{equation}

If $p<n/(n-2)$, combining \eqref{eqi1} and \eqref{eqi2}, and letting $R\to \infty$ yields a contradiction.

Finally, assume that $p=n/(n-2)$, that is, $2/(p-1) = n-2$.
Set $\tilde u_R (x) = R^{n-2} u(Rx)$. Then
\begin{equation}\label{eqi3}
-\Delta \tilde u_R \geq V(Rx) \tilde u_R^p.
\end{equation}
Observe that
$$
\tilde m(R) := \inf_{\partial B_1} \tilde u_R = \inf_{\partial B_R} \frac{u}{\Phi}, $$
where $\Phi(x) = |x|^{2-n}$. We proved in \eqref{eqi1}--\eqref{eqi2} that $0<c\leq \tilde m(R) \leq C$, for $R\ge R_0$.

By the maximum principle $ u(x)\geq \tilde m(R) \Phi(x)$ in $\rn\setminus B_R$, which is equivalent to  $\tilde u_R\geq \tilde m(R) \Phi$ in $\rn\setminus B_1$, by the $(2-n)$-homogeneity of $\Phi$. In addition, $\tilde m$ is nondecreasing in $R$.

So \eqref{eqi3} implies
\begin{equation}\label{eqi4}
-\Delta (\tilde u_R - \tilde m(R) \Phi) \geq V(Rx) \tilde u_R^p \geq cV(Rx)\quad\mbox { in }\; B_5\setminus B_{1}.
\end{equation}

We apply Lemma \ref{lem_qsmp} to this inequality, with $\Omega = B_5\setminus B_{1}$ and $K = B_4\setminus B_{3/2}$, to deduce that
$$
\tilde u_R \geq \tilde m(R) \Phi + c_0 = (\tilde m(R) + c_0 2^{n-2}) \Phi \qquad \mbox{on }\; \partial B_2,
$$
that is,
$$
u \geq (\tilde m(R) + c_0 2^{n-2}) \Phi \qquad \mbox{on }\; \partial B_{2R}.
$$
Hence
$$
\tilde m(2R) \geq \tilde m(R) + c_0 2^{n-2},
$$
which implies $\tilde m(R)\to \infty$ as $R\to \infty$, a contradiction.
\end{proof}
\smallskip

The following lemma is a generalization of a classical result of Keller and Osserman to weak solutions of coercive problems with weights.

\begin{lem}
\label{eigenvalue_V2}
Let $W$  be a nonnegative weak solution  of
\begin{equation}
\label{equ_lin}
\Delta W \geq \frac{A}{1+|x|^2}\, W^p \qquad \mbox{in }\; \mathbb{R}^n,
\end{equation}
where $p\geq 0$ and $A>0$.
\begin{itemize}
\item[(i)] If $W\in L^\infty(\rn)$, then $W=0$.
\item[(ii)] If $p>1$, then $W=0$.
\end{itemize}
\end{lem}

The statement (ii) in this lemma appeared first in \cite{Lin2}
 {(see also \cite{NiWM} for an earlier result for potentials with subquadratic decay).}
 We will provide a full and simplified proof  in the case of the Laplacian, for the reader's convenience.

\begin{proof}[\underline{Proof of Lemma \ref{eigenvalue_V2}}] In what follows we denote $$V(x) = \frac{A}{1+|x|^2}.$$

{\it Step 1}.  We assume $p>1$ or $W\in L^\infty(\rn)$. Given a solution $W$ of $(\ref{equ_lin})$, we prove the existence of a
smooth $Z$
satisfying the same equation, with possibly modified constant $A$.

We pick a nonnegative $\rho\in C^\infty(\rn)$ with support inside
$B(0,1)$ such that $\int_\rn \rho =1$ and set $Z=W*\rho\in C^\infty(\rn)$. It is easy to
see that $$\Delta Z \geq [V\,W^p]*\rho$$ in the classical sense.
 Note that if $|y|\leq 1$, then $V(x-y)\geq
\frac{1}{2}\frac{A}{1+|x|^2} $.

So $$[V\,W^p]*\rho(x)=\int_\rn V(x-y)W^p(x-y)\rho(y)\,dy
\geq \frac{C}{2}\frac{A}{1+|x|^2}  Z^p,$$ where
$C=1$ if $p\geq 1$
 and  $C=\frac{1}{\|W\|_\infty^{1-p}}$  if $0\leq p<1$
(if $p>1$ we use Jensen's inequality).

Hence, $$\Delta Z \geq \tilde{V} Z^p$$ where
$\tilde{V}=\frac{\tilde{A}}{1+|x|^2}$ and $\tilde{A}=\frac{CA}{2}$.
Note that if $W\in L^\infty(\rn)$, then $Z\in L^\infty(\rn)$.
\smallskip

{\it Step 2}. From Step 1, we can assume that $W$ is smooth.

(i) Suppose for contradiction that $W\geq 0$ is bounded on $\mathbb{R}^n$ and does not vanish
identically. We can   assume without loss of generality that
$W(0)>0$, since the problem is invariant with respect to  translations (a translation of $V$ gives a function whose behaviour is the same as $V$).

For any $R>0$, we denote the spherical mean of $W$ by
$$\overline{W}(R)=\frac{1}{|S_R|}\int_{S_R} W \;d\sigma_R$$
 where $S_R$ is the sphere of center 0 and radius R, $\sigma_R$ is the
Lebesgue's measure on $S_R$ and $|S_R|=\sigma_R(S_R)$. It is clear that
$\overline{W}$ is bounded on $(0,+\infty)$.

Since $W$ is subharmonic,
$$ W(0)\leq \frac{1}{|B_R|}\int_{B_R} W \; dx . $$
It is easy to see that
there exists $C>0$ independent of $R$ such that {
$$ (W(0))^{\max(p,1)}\leq C\;\frac{1}{|B_R|}\int_{B_R} W^p \; dx . $$
Indeed, if $p\ge1$ then  this is a consequence of Jensen's inequality (and $C=1$), whereas if
$0\leq p\leq 1$, we can use the boundedness of $W$
(and $C=\|W\|_\infty^{1-p}$).}

 We know that
 $$ \frac{d\overline{W}}{dR}=\frac{1}{|S_R|}\int_{B_R} \Delta W \;d\sigma_R,$$
 hence
 $$\frac{d\overline{W}}{dR}\geq \frac{A}{n}\frac{R}{1+R^2}\frac{1}{|B_R|}\int_{B_R} W^p
 \; dx \geq \frac{A}{nC}
\frac{R}{1+R^2} (W(0))^{\max(p,1)} = C\frac{R}{1+R^2}.$$
 But this implies  $\overline{W}(R)\underset{R\rightarrow +\infty}{\longrightarrow}
+\infty$, which is a contradiction.  \smallskip

 (ii) We assume $p>1$ and will prove that $W$ is bounded, which  implies the result,
 by the statement (i).

Arguing as in  \cite{Osserman}, we define the function $W_R$ on $B_R$ by
 $$W_R(x)=C \frac{R^{2\alpha}}{(R^2-|x|^2)^\alpha}, $$
 where $\alpha=\frac{2}{p-1}$. We will see, by direct computation, that if $C>0$ is
large enough, then
 \begin{equation}
\label{equation_W_R}
 \Delta W_R \leq \frac{A}{1+|x|^2}{W_R}^p.
 \end{equation}
 Indeed, denoting $r=|x|$, we have
 $$\Delta W_R=2\alpha C R^{2\alpha}\;\frac{n(R^2-r^2)+2(\alpha+1)r^2}
{(R^2-r^2)^{\alpha+2}}\leq 2\alpha C R^{2\alpha+2}\;\frac{n+2(\alpha+1)}
{(R^2-r^2)^{\alpha+2}}$$
 and
$$\frac{A}{1+r^2}{W_R}^p=\frac{A}{1+r^2} \frac{C^pR^{2\alpha p}}{(R^2-r^2)^{\alpha p}}
\geq  \frac{A}{1+R^2} \frac{C^pR^{2\alpha p}}{(R^2-r^2)^{\alpha p}}.$$
We note that $\alpha+2=\alpha p$.
Hence, a sufficient condition to have (\ref{equation_W_R}) is
$$ C^{p-1}\geq (1+R^2)  R^{2\alpha+2-2\alpha p}\frac{2\alpha[n+2(\alpha+1)]}{A}.$$
Since $2\alpha+2-2\alpha p = -2$,  for each $R\ge 1$ a sufficient condition for the last inequality is
$$ C^{p-1}\geq \frac{4\alpha[n+2(\alpha+1)]}{A},$$
and this is how we choose $C$.

It is now easy to see that $W\leq W_R$ on $B_R$. Note that $W_R(x)\to\infty$ as $x\to\partial B_R$. If we denote
$w=W-W_R$ and if S is a $C^2$ nondecreasing convex function on
$\mathbb{R}$ such that $S=0$ on $(-\infty,0]$ and $S>0$ otherwise, then
$$\Delta
S(w)\geq S'(w) \Delta w\geq S'(w)V(x)( W^p-{W_R}^p)
\geq 0.$$ Hence $S(w)$ is subharmonic on $B_R$ and can
be continuously extended on $\overline{B_R}$ by setting $S(w)=0$ on $S_R$, so
 by the maximum principle $S(w)=0$, that is, $w\leq 0$.

Finally, for all $R\geq 1$, $W\leq W_R$ on $B_R$,  so by letting $R\to \infty$
 we obtain $W(x)\leq \lim_{R\to\infty} W_R(x)= C$, for each $x\in \rn$.
 \end{proof}

\mysection{Proofs of the classification and Liouville theorems in the whole space }

\subsection{Proof of Theorem \ref{thm_Spqr}.}

The key idea is to use the two auxiliary functions
$$W:=|u-Kv|$$
and
$$Z:=\min(u,Kv),$$
where $K$ is given by Proposition~\ref{Prop-uKv}. Clearly $u\equiv Kv$ is equivalent to $W\equiv0$, and $u=v\equiv 0$ is equivalent to $Z\equiv 0$, when $K>0$.

The following two lemmas assert that the functions $ Z,W$ satisfy a suitable
system of elliptic inequalities, respectively of the form (\ref{eqVBR}) and (\ref{equ_lin}).

\begin{lem}
\label{minuv}
We suppose that \eqref{Hyp_abcdpqr} holds.

\smallskip
(i) Assume $ ab \geq cd$. Then $Z$ is superharmonic.

If $p+q< 1$, suppose in addition that $(u,v)$ is bounded.
Then $Z$ is a weak solution of
\begin{equation}
\label{DeltaZineq}
-\Delta Z \ge CW^\beta Z^r\qquad\mbox{in }\; \rn,
\end{equation}
where $\beta:=\max(p+q,1)$ and $C>0$.

\smallskip
(ii) Assume $ab>cd$. Then $Z$ is a weak solution of
$$-\Delta Z \ge CZ^{p+q+r}\qquad\mbox{in }\; \rn.$$
\end{lem}

\begin{lem}
\label{DeltaW}
We suppose that \eqref{Hyp_abcdpqr} holds, and $ ab \geq cd$.

\smallskip
(i) Then $W$ is subharmonic.

\smallskip
(ii) Assume $r>p$ and $c, d>0$.
We also suppose that $(u,v)$ is bounded in case $q+r< 1$.
Then $W$ is a weak solution of
\begin{equation}
\label{DeltaWineq}
\Delta W \ge CZ^p W^\gamma\qquad\mbox{in }\; \rn,
\end{equation}
where $\gamma:=\max(q+r,1)$ and $C>0$.
\end{lem}

\begin{proof}[\underline{Proof of Lemma~\ref{minuv}}]
Let us recall the Kato inequality (valid in particular for  weak solutions):
\begin{equation}
\label{KatoIneq}
\Delta z_+\ge \chi_{\{z>0\}} \Delta z.
\end{equation}

(i) Writing
$$Z={1\over 2}\bigl(u+Kv-(u-Kv)_+-(Kv-u)_+\bigr),$$
it follows from (\ref{KatoIneq}) that
$$-\Delta Z \ge {1\over 2}\bigl(-\Delta(u+Kv)+\chi_{\{u>Kv\}} \Delta(u-Kv)+\chi_{\{u<Kv\}} \Delta(Kv-u)\bigr),$$
hence
\begin{equation}
\label{KatoIneq2}
-\Delta Z \ge -\chi_{\{u<Kv\}} \Delta u - K\chi_{\{u>Kv\}} \Delta v-  {1\over 2}\chi_{\{u=Kv\}} \Delta (u+Kv).
\end{equation}

Now we make use of the inequality
$$x^q-y^q\ge C_q x^{q-1}(x-y),\quad x>y>0$$
with $C_q=1$ if $q\ge 1$, $C_q=q$ if $0<q<1$.
By Proposition~\ref{Prop-uKv}, we have
\begin{equation}
\label{abcdK}
a-cK^q\ge 0,\quad bK^q-d\ge 0.
\end{equation}
Therefore, on the set $\{u\le Kv\}$, we obtain
\begin{equation}
\label{InequleKv}
\alignedat2
-\Delta u&=u^rv^p(av^q-cu^q)
 \ge aK^{-q}u^rv^p((Kv)^q-u^q) \\
 &\ge aC_qK^{-1}u^rv^{p+q-1}(Kv-u) \ \ge 0.
 \endalignedat
\end{equation}
  Similarly, on the set $\{u\ge Kv\}$, we get
\begin{equation}
\label{InequgeKv}
\alignedat2
-\Delta v&=v^ru^p(bu^q-dv^q)
 \ge bv^ru^p(u^q-(Kv)^q) \\
&\ge bC_qv^ru^{p+q-1}(u-Kv) \  \ge 0.
 \endalignedat
\end{equation}
{In particular}, $-\chi_{\{u=Kv\}} \Delta (u+Kv)\geq 0$. 
  Hence, we deduce from (\ref{KatoIneq2}) that
  \begin{equation}
\label{KatoIneq2b} 
-\Delta Z \ge -\chi_{\{u<Kv\}} \Delta u - K\chi_{\{u>Kv\}} \Delta v,
\end{equation}
  so $Z$ is superharmonic, by (\ref{InequleKv}) and (\ref{InequgeKv}).

   Now assume either that $p+q\ge 1$ or that $(u,v)$ is bounded.
By using that
$$v^{p+q-1}\ge C(Kv-u)^{p+q-1}\quad\mbox{
  if }\;p+q\ge 1,$$
  and $v^{p+q-1}\ge C>0$ otherwise (since $v$ is bounded),    we infer from (\ref{InequleKv}) and (\ref{InequgeKv}) that
  $$-\Delta u \ge Cu^r(Kv-u)^\beta\quad\hbox{ on }\;\{u\le Kv\},$$
  and
  $$-\Delta v \ge  C(Kv)^r(u-Kv)^\beta \quad\hbox{ on }\;\{u\ge Kv\}.$$

We then deduce from {(\ref{KatoIneq2b})} that 
$$-\Delta Z \ge Cu^r(Kv-u)^\beta_+ + C(Kv)^r(u-Kv)^\beta_+=C|u-Kv|^\beta Z^r.$$

(ii) If $ab>cd$, then the inequalities in (\ref{abcdK}) are strict, that is $a\ge cK^q+\eps$, $bK^q\ge d+\eps$ for some $\eps>0$.  Then we obtain, as in (\ref{InequleKv}) and (\ref{InequgeKv}), that
$$-\Delta u \ge \eps u^rv^{p+q}\ge \eps K^{-p-q} Z^\sigma\qquad \mbox{on the set}\quad\{u\le Kv\},$$
and
$-\Delta v\ge \eps {u^pv^{q+r}}\ge \eps K^{{-q-r}} Z^\sigma$ on the set $\{u\ge Kv\}$,
for some $\eps>0$.
The assertion then follows from {(\ref{KatoIneq2b}).} 
\end{proof}
\smallskip

\begin{proof}[\underline{Proof of Lemma~\ref{DeltaW}}]
(i) By using (\ref{KatoIneq}) and Proposition~\ref{Prop-uKv}, we get
$$
\begin{aligned}
\Delta W &= \Delta(u-Kv)_+ + \Delta(Kv-u)_+ \\
&\ge \chi_{\{u>Kv\}} \Delta(u-Kv)+\chi_{\{u<Kv\}} \Delta(Kv-u) \\
 \end{aligned}
 $$
{hence  
\begin{equation}
\label{ineqDeltaW}
\Delta W\ge\chi_{\{u>Kv\}} (Kg-f) +\chi_{\{u<Kv\}} (f-Kg) \ge 0,
\end{equation}}
where we have set $f(u,v) = u^rv^p[av^q-cu^q]$, $g(u,v) = v^ru^p[bu^q-dv^q]$.
\smallskip

(ii) In Lemma \ref{elemlem}(i) in the appendix we show that
$$
(Kg-f)(u-Kv)\geq C u^p v^p(u+Kv)^{q+r-p-1} (u-Kv)^2,
$$
{when $r>p$ and $c, d>0$.}  
{Using (\ref{ineqDeltaW}), we then get}  
$$
\begin{aligned}
\Delta W
&\ge\chi_{\{u>Kv\}} (Kg-f) +\chi_{\{u<Kv\}} (f-Kg)\\
\noalign{\vskip 1mm}
&\ge Cu^pv^p(u+Kv)^{q+r-p-1}|u-Kv| \\
\noalign{\vskip 1mm}
&\ge C_1Z^p(u+Kv)^{q+r-1}|u-Kv|.
\end{aligned}
$$
If $q+r\ge 1$, we conclude by using $(u+Kv)^{q+r-1}\ge |u-Kv|^{q+r-1}$.
If $q+r< 1$, we conclude by using
$(u+Kv)^{q+r-1}\ge C$, in view of the boundedness of $(u,v)$.
\end{proof}

\begin{proof}[\underline{Proof of Theorem \ref{thm_Spqr}}]
 (i) Assume for contradiction that $u\not\equiv Kv$.
By Lemma~\ref{DeltaW}(i), the function $W=|u-Kv|$ is subharmonic, nonnegative and nontrivial. Clearly, so is $W^\beta$, for each $\beta\ge1$. Then Lemma \ref{eigenvalue_V}  applies to the inequality $-\Delta Z\ge W^\beta Z^r$, which we proved in Lemma \ref{minuv} (recall the discussion after the statement of Lemma \ref{eigenvalue_V}). Hence $Z\equiv 0$, a contradiction.

\smallskip
 (ii) First, we observe that we may assume  $q+r> 1$. Indeed, if  $q+r\le 1$, then $(u,v)$ is assumed to be bounded
and, since $r\le q+r\le 1<  n/(n-2)_+$, the conclusion follows from assertion~(i).

Next we claim that we may assume $r>p$.
Indeed, if $p+q<1$, then this is true due to $q+r>1>p+q$.  If $p+q\ge 1$, then we may assume $r>n/(n-2)$ and $n\ge 3$,
since otherwise the result is already known from assertion~(i). But then $r>2/(n-2)\ge p$.

Now, by Lemma~\ref{minuv}(i), $Z$ is superharmonic and positive, hence
$$Z(x)\ge c_1 (1+|x|)^{2-n},\quad x\in\mathbb{R}^n,$$ 
for some $c_1>0$. Therefore
$$Z^p(x)\ge \tilde c_1 (1+|x|)^{-(n-2)p}\ge \tilde c_1 (1+|x|)^{-2},\quad
x\in\mathbb{R}^n.$$
Hence we can  apply Lemma~\ref{eigenvalue_V2} to the inequality $\Delta W \geq Z^pW^\beta$, which we obtained in Lemma~\ref{DeltaW}(ii), and  conclude that $W\equiv 0$.
\end{proof}

\begin{rmq}
\label{rmqCounterIneq}
It does not seem possible to go beyond assumptions (\ref{HypthmSpqr}), (\ref{HypthmSpqr2}) by the sole means of the mixed-type system
(\ref{DeltaZineq})-(\ref{DeltaWineq}).
Indeed,  if $n\geq 3$, $r>\frac{2}{n-2}$ and $p>\frac{2}{n-2}$, then this system admits positive solutions of the form 
$$Z=C(1+|x|^2)^{-\alpha},\quad W=B-A(1+|x|^2)^{-\beta},$$
with suitable $B>A>0$, $C>0$,
 $2/(n-2)<1/\alpha<\min(p, r-1)$ and $0<\beta<p\alpha-1$
(this is easily checked by direct computation).
We remark that Keller-Osserman type estimates and Liouville theorems for another mixed-type system, namely
$$-\Delta Z = W^p, \quad \Delta W=Z^q,$$
were obtained in the recent work \cite{BVGY}.
\end{rmq}

\begin{proof}[\underline{Proof of Theorem \ref{thm_SpqrLiouv}}]\quad
Assume first that $\sigma\le n/(n-2)_+$.
Then the result is a consequence of Lemma~\ref{minuv}(ii) and Lemma
\ref{eigenvalue_V}.

Assume next that $n\ge 3$ and $\sigma>n/(n-2)$. 
Suppose for contradiction that a positive bounded solution $(u,v)$ exists.
By  (\ref{Hyp_abcdpqr}), we have $r\le p+q$, hence $r\le \sigma/2$.
Since $\sigma<(n+2)/(n-2)$, we deduce $r\le n/(n-2)$.
Theorem~\ref{thm_Spqr} then guarantees that $u=Kv$,
where $K$ is given by Proposition~\ref{Prop-uKv}.
It follows that
$$-\Delta v=K^{-1}u^rv^p(av^q-cu^q)=Cv^\sigma,\quad x\in\mathbb{R}^n, $$
with $C=K^{r-1}(a-cK^q)>0$ by Proposition~\ref{Prop-uKv}.
But this contradicts a well-known Liouville-type result from \cite{GS2}.

Moreover, if either $p+q\geq 1$, or $p\le 2/(n-2)$ (hence $q+r>1$ due to $\sigma>n/(n-2)$),
then the boundedness assumption is not necessary when applying Theorem~\ref{thm_Spqr}.
Finally, we note that if $n\le 4$, then we always have
either $p+q\ge 1$ or $p<1\le 2/(n-2)$.
\end{proof}

\begin{proof}[\underline{Proof of Proposition~\ref{propo_cdzero}}]
We may assume without loss of generality that $d=0$.
(Indeed the system \eqref{system_Spqr} with unknown $(u,v)$, parameters $a, b, c, d$ and exponents $p, q, r$ is equivalent to
the system \eqref{system_Spqr} with unknown $(v,u)$, parameters $b, a, d, c$ and same exponents.)
 Also we may assume that $c>0$ since, in the case $c=d=0$, the result is already known from Theorems~1.4(i) and 1.2 in \cite{QS}.
(This is actually proved there in the case $a=b=1$, but the general case immediately follows by scaling.)

Now assume that $(Kv-u)_+\not\equiv 0$.
Since
$$\Delta(Kv-u)_+ \ge \chi_{\{u<Kv\}} \Delta(Kv-u) \ge \chi_{\{u<Kv\}} (f-Kg) \ge 0,$$
due to Proposition~\ref{Prop-uKv}, the function $(Kv-u)_+$ is subharmonic.
It follows (see the discussion after the statement of Lemma \ref{eigenvalue_V})) that
$$\liminf_{R\to\infty} R^{-n}\int_{B_R} (Kv-u)_+(x)\, dx>0.$$
Consequently, since $v\ge (1/K)(Kv-u)_+$ we have $\liminf_{R\to\infty}\overline{v}(R)=:L>0$,
where $\overline{v}(R)=|S_R|^{-1}\int_{S_R}v\,d\sigma_R$ denote the spherical means.
But, since $v$ is superharmonic  due to $d=0$, 
we deduce from \cite[Lemma 3.2]{QS} that
\begin{equation}
\label{VgeL}
v\geq L>0,\quad x\in\mathbb{R}^n.
\end{equation}

On the other hand,  by Lemma~\ref{elemlem}(ii), we have
$$
\begin{aligned}
(Kg-f)(u-Kv)
&\geq C u^r v^{p\wedge r}(u+Kv)^{q-1+(p-r)_+} (u-Kv)^2 \\
&\geq C v^{p\wedge r}|u-Kv|^{q+1+(p\vee r)}.\\
\end{aligned}
$$
Therefore,
$$\Delta (u-Kv)_+ \geq \chi_{\{u>Kv\}} (Kg-f) \geq C (u-Kv)_+^{q+(p\vee r)},$$
owing to (\ref{VgeL}). In view of Lemma~\ref{DeltaW}(ii), we conclude that $u\leq Kv$.
\end{proof}
\medskip

Finally, let us  justify the statement in Remark~\ref{rmqRad}. Let $W_1:=(u-Kv)_+$ and $W_2:=(Kv-u)_+$.
By the proof of Lemma~\ref{DeltaW}(i), we know that the radially symmetric functions
$W_1$ and $W_2$ are subharmonic, hence (radially) nondecreasing.
Since $W_1W_2\equiv 0$, we necessarily have $\lim_{t\to\infty} W_1(t)=0$ or $\lim_{t\to\infty} W_2(t)=0$,
hence $W_1\equiv 0$ or $W_2\equiv 0$.

\subsection{Proof of Theorem \ref{thm_SpqrCounter}}

By adding up the two equations, we see that $u+v$ is harmonic and positive, hence constant. We therefore look
for a solution such that $v=1-u$, with $0<u<1$.
The system then becomes equivalent to
\begin{equation}
\label{PDEivp}
-\Delta u = u^p(1-u)^p[(1-u)^q-u^q] =: f(u).
\end{equation}
To show the existence of a nonconstant positive solution of (\ref{PDEivp}), we argue like in the proof of \cite[Theorem 1.4]{QS}.
Consider the initial value problem for the real function $u=u(t)$
\begin{equation}
\label{IVP}
-(t^{n-1}u')'=t^{n-1}f(u),\ \ t>0,\qquad u(0)=\eps, \quad u'(0)=0,
\end{equation}
with $0<\eps<\frac{1}{2}$. It is standard to check that either $u > 0$, $u'\le 0$ for all $t > 0$, or $u$ has a first zero $t=R$.
If the latter occurs, then the PDE in (\ref{PDEivp}) admits a positive solution $u$ in a finite ball with homogeneous Dirichlet conditions,
and also $0<u<\eps$.
But this is known to be impossible, owing to the Pohozaev identity, whenever
\begin{equation}
\label{Pohoz}
h(X):=Xf(X) - (p_S+1)F(X)\ge 0,\quad 0<X<\eps,
\end{equation}
where $F(X)=\int_0^X f(\tau)\,d\tau$ and $p_S=\frac{n+2}{n-2}$.
In the case of (\ref{PDEivp}) we have $h(0)=0$ and
$$h'(X)=Xf'(X)-p_Sf(X)\sim (p-p_S)X^p>0,\quad\hbox{as $X\to 0^+$}.$$
Therefore (\ref{Pohoz}) is true for $\eps>0$ sufficiently small, and the conclusion follows.

\mysection{Properties of half-spherical means. Proofs of the classification results in a half-space. }

{We start by recalling that the classical Phragmén-Lindel\"of maximum principle  states that a subharmonic function with sublinear growth in the half-space which is nonpositive on $\partial\mathbb{R}^n_+$ is also nonpositive in $\mathbb{R}^n_+$  (see for instance \cite{prowei}).

\begin{proof}[\underline{Proof of Theorem \ref{thm_general}}] \quad
We set $w=u-Kv$. Since $u,v$ have sublinear growth, so do  $|w|$ and  $w_+$.
Let $\psi\in C^2(\mathbb{R} )$ be convex, nondecreasing and such that
$0\leq \psi(t)\leq t^+$ for all $t\in \mathbb{R}$ and $\psi(t)>0$ for $t>0$. Then
$\psi(w)$ has sublinear growth.\\
Since the nonlinearities satisfy condition (\ref{condition_fg}), then  $w\geq 0$
implies $\Delta w\geq 0$.
Hence, we have $$\Delta \psi(w)=\psi'(w)\Delta w+\psi''(w)|\nabla w|^2\geq 0,$$
since
$\psi'(w)=0$ if $w\leq 0$ and $\Delta w\geq 0$ otherwise. Hence, $\psi(w)$ is subharmonic.
Since by hypothesis $w=0$ on $\partial \mathbb{R}^n_+$, $\psi(w)=0$ on $\partial \mathbb{R}^n_+$. By the
Phragmén-Lindel\"of maximum principle, we get $\psi(w)\leq 0$, so $w\leq 0$.
The same argument applied to $-w$ leads to $-w\leq 0$. Finally, we obtain $w=0$,
 i.e. $u=Kv$.
\end{proof}}

We will use the following notation: for any $R>0$, and any $y\in \partial \mathbb{R}^n_+$, we set
$$S_R^+(y)=\{x\in \mathbb{R}^n_+,\, |x-y|=R \},$$
$$B_R^+(y)=\{x\in \mathbb{R}^n_+,\, |x-y|\leq R \},$$
$$D_R(y)=\{y'\in \partial \mathbb{R}^n_+,\, |y'-y|\leq R \},$$
and write $S_R^+$, $B_R^+$ and $D_R$ respectively for
$S_R^+(0)$, $B_R^+(0)$ and $D_R(0)$.
We recall  the definition of the half-spherical means of a function $w$, namely
$$
[w]_y(R)=\frac{1}{R^2 |S_R^+|}\int_{S_R^+(y)} w(x)\;x_n\,d\sigma_R(x),\quad R>0,\ y\in \partial \mathbb{R}^n_+,
$$
and  denote $[w]:=[w]_0$.
Observe that $[x_n]$ is a positive constant (independent of $R$).
\smallskip

The following lemma provides a basic computation for the derivative of the half-spherical mean with respect to the radius.

 \begin{lem}
\label{lem_derivee_moyenne}
Let $u\in C^2(\overline{\mathbb{R}^n_+})$. For any $y\in \partial \mathbb{R}^n_+$ and $R>0$, we have :
$$\frac{d}{dR} [u]_y(R)=\frac{1}{R^2 |S_R^+|}\left[ \int_{B_R^+(y)} \Delta u\;x_n\,dx-\int_{D_R(y)}
u(y') \,dy' \right].$$
\ \end{lem}

\begin{proof} We have
\begin{align*}
I=&\int_{B_R^+(y)} \Delta u \; x_n\,dx-\int_{D_R(y)}
u(y') \,dy'=\int_{B_R^+(y)}div(\nabla u \; x_n-u\, e_N) \,dx-\int_{D_R(y)}
u(y')\, dy'\\
=& \int_{S_R^+(y)} \nabla u\cdot\vec{\nu}\; x_n \,d\sigma_R(x)-
\int_{S_R^+(y)} u\; \nu _n \,d\sigma_R(x),
\end{align*}
since $x_n=0$ on $\partial \mathbb{R}^n_+$.
Then, setting $x=y+Rz$,
since $\nu(x)=z$ and $\nu_n(x)=\frac{x_n}{R}$, we have
\begin{align*}
 I&=R^n\int_{S_1^+} \nabla u(y+Rz)\cdot z\; z_n \;d\sigma_1(z)-\frac{1}{R}\int_{S_R^+(y)} x_n u
 \; d\sigma_R(x)\\
&=
R^n\frac{d}{dR} \int_{S_1^+} u(y+Rz)\, z_n \;d\sigma_1(z)-R|S_R^+|[u]_y(R)\\
&= R^n\frac{d}{dR} \bigl(R\,|S_1^+| [u]_y(R)\bigr)-R|S_R^+|[u]_y(R)\\
&=R|S_R^+|\Bigl( \frac{d}{dR}\bigl(R[u]_y(R)\bigr)-[u]_y(R)\Bigr)=R^2|S_R^+|\frac{d}{dR} [u]_y(R).
\end{align*}
\end{proof}

Next, we give a the generalization of the Phragmén-Lindel\"of maximum principle,
based on the above monotonicity property, which will play an important role.

\begin{lem}
\label{lemme_principal}
 Let $w\in C^2(\overline{\mathbb{R}^n_+})$ be such that $w\leq 0$ on $\partial \mathbb{R}^n_+$ and
$\Delta w\geq 0$ on the set~$\{w>0\}$.
If we assume
\begin{equation}
 \label{condition_limite_lem_principal}
\underset{R \rightarrow \infty }{\lim\inf} [w^+](R)=0,
\end{equation}
 then $w\leq 0$ in $\mathbb{R}^n_+$.

\end{lem}

\begin{proof}[\underline{Proof}] Let $\psi\in C^2(\mathbb{R} )$ be convex, nondecreasing and such that
$0\leq \psi(t)\leq t^+$ for all $t\in \mathbb{R}$ and $\psi(t)>0$ for $t>0$. Then, for any $R>0$,
$0\leq [\psi(w)](R)\leq [w^+](R)$. Therefore,
\begin{equation}
\label{decaypsi2}
 \underset{R \rightarrow \infty }{\lim\inf} [\psi(w)](R)=0.
 \end{equation}
We also have $$\Delta \psi(w)=\psi'(w)\Delta w+\psi''(w)|\nabla w|^2\geq 0,$$
since
$\psi'(w)=0$ if $w\leq 0$ and $\Delta w\geq 0$ otherwise.
Since $w\leq 0$ on $\partial \mathbb{R}^n_+$, then $\psi(w)=0$ on $\partial \mathbb{R}^n_+$ so
Lemma \ref{lem_derivee_moyenne} gives
 that $[\psi(w)](R)$ is nondecreasing.
But its limit as $R\rightarrow \infty$ is zero by
(\ref{decaypsi2}), so $[\psi(w)](R)=0$ for all $R>0$. This implies that
$\psi(w)\equiv 0$, hence $w\leq 0$.
\end{proof}

{

The proof of Theorem \ref{thm_general2} is  an easy  consequence of
 Lemma~\ref{lemme_principal}.

\begin{proof}[\underline{Proof of Theorem \ref{thm_general2}}]
Let $w=u-Kv$. Observe that $w=0$ on $\partial \mathbb{R}^n_+$ and $w\,\Delta w\geq 0$
 thanks to (\ref{condition_fg}). Hence we can apply Lemma \ref{lemme_principal}
to $w$ and $-w$ and conclude that $w=0$.
\end{proof}
}

\medskip

 We now turn to the proof of Theorem \ref{thm_systeme_S_rsd_generalise}.
In order to treat solutions without growth restrictions at infinity in the case of positive nonlinearities,
we will need to exploit some further properties of half-spherical means for superharmonic functions.

 The following lemma will permit to us to split the proof of Theorem \ref{thm_systeme_S_rsd_generalise} in the following way:
 either the superharmonic functions $u,v$ grow at infinity at least like $x_n$ and then we  apply the nonexistence result for weighted inequalities in Lemma \ref{lem_estimation_integrale_Liouville}, or the  half-spherical means of $u,v$ decay at infinity and we can use  Theorem \ref{thm_general2}.
\begin{lem}
\label{lem_lambda}
Suppose that $u\in C^2(\overline{\mathbb{R}^n_+})$ is nonnegative and superharmonic in $\mathbb{R}^n_+$.
\begin{itemize}
\item[(i)] For each $y\in \partial \mathbb{R}^n_+$, the function $R\mapsto [u]_y(R)$ is nonincreasing and its limit is independent of $y$.
\item[(ii)] Denote $L(u):=\displaystyle\lim_{R\to\infty} [u](R) \in [0,\infty)$.
Then we have
$$u(x)\geq \frac{L(u)}{[x_n]}\,x_n,\quad  x\in \mathbb{R}^n_+.$$
\end{itemize}
 \end{lem}

Assertion (ii) can be deduced from a more general and rather difficult result from~\cite{Gardiner};
see Remark~\ref{rmqGardiner} below. We will provide a direct, more elementary proof.

\begin{proof}[\underline{Proof}]
(i) That   $[u]_y(R)$ is nonincreasing in $R$ is a direct consequence of Lemma \ref{lem_derivee_moyenne}.
Set
\begin{equation}
 \label{defmu}
\mu(y):=\lim_{R\to\infty}R^{-(n+1)}\int_{S_R^+(y)}x_n u\, d\sigma_R = |S_1^+| \lim_{R\to\infty} [u]_{y}(R).
\end{equation}
By L'H\^opital's rule, (\ref{defmu}) implies that
$$\lim_{R\to\infty}R^{-(n+2)}\int_{B_R^+(y)}x_n u \,dx
=\lim_{R\to\infty}R^{-(n+2)}\int_0^R\int_{S_r^+(y)}x_n u\, d\sigma_r\,dr = {\mu(y)\over n+2}$$
(with nonincreasing limit).
Now, for $y_1,y_2\in \partial \mathbb{R}^n_+$, we have $B_R^+(y_1)\subset B_{R+|y_1-y_2|}^+(y_2)$,
hence
$$R^{-(n+2)}\int_{B_R^+(y_1)}x_n u \,dx \le (1+R^{-1}|y_1-y_2|)^{n+2}(R+|y_1-y_2|)^{-(n+2)}\int_{B_{R+|y_1-y_2|}^+(y_2)}x_n u \,dx.$$
By letting $R\to\infty$, we deduce that $ \mu(y_1)\le  \mu(y_2)$,
which proves that $\mu(y)$ is independent of $y$.

\medskip

(ii) The proof is divided in three steps.
\smallskip

{\it Step 1.} We recall several properties of Poisson kernels, {that is, normal derivatives of Green functions}. For $R>0$, we denote by $P_R(x;y)$ the Poisson kernel of $B_R^+$.
Then for any $\varphi\in C(\partial B_R^+)$, the unique harmonic function $v$ in $B_R^+$ with boundary value $\varphi$
is given by
$$v(x)=\int_{\partial B_R^+}P_R(x;y)\varphi(y)\, d\sigma_R(y).$$
A simple rescaling argument shows that
\begin{equation}
 \label{scalingPoisson}
P_R(x;y)=R^{1-n}P_1(R^{-1}x;R^{-1}y).
\end{equation}
On the other hand, for each $Y\in \partial B_1^+$, $P_1(\cdot,Y)$ is positive in $B_1^+$ (by the strong maximum principle,
since it is harmonic, nonnegative and nontrivial).
For each $X\in B_1^+$, since $Y\mapsto P_1(X;Y)$ is continuous on $\partial B_1^+$, it follows that
\begin{equation}
 \label{infPoisson}
 c(X):=\inf_{Y\in \partial B_1^+} P_1(X;Y)>0.
\end{equation}

\medskip

{\it Step 2.} Fix $x\in H$, denote by $\tilde x=(x_1,\cdots,x_{n-1},0)$ its projection onto $\partial \mathbb{R}^n_+$ and set $R=2x_n$.
Since $u(\tilde x+\cdot)\ge 0$ is superharmonic in $B_R^+$,
the maximum principle implies that, for all $z\in B_R^+$,
$$u(\tilde x+z)\ge\int_{\partial B_R^+}P_R(z;y)u(\tilde x+y)\, d\sigma_R(y)
\ge\int_{S_R^+}P_R(z;y)u(\tilde x+y)\, d\sigma_R(y),$$
hence
$$u(\tilde x+z)\ge R^{1-n}\int_{S_R^+}P_1(R^{-1}z;R^{-1}y)u(\tilde x+y)\, d\sigma_R(y),$$
due to (\ref{scalingPoisson}).

Now take $z=(0,\cdots,0,x_n)$, set $X=(0,\cdots,0,1/2)$ and $c_0=c(X)$ (see~(\ref{infPoisson})).
Using (\ref{infPoisson}) and assertion~(i), we obtain
$$
\begin{aligned}
u(x)&\ge R^{1-n}\int_{S_R^+}P_1(X;R^{-1}y)u(\tilde x+y)\, d\sigma_R(y)
\ge c_0R^{1-n}\int_{S_R^+}u(\tilde x+y)\, d\sigma_R(y)\\
&= c_0R^{1-n}\int_{S_R^+(\tilde x)}u(y)\, d\sigma_R(y)
\ge c_0R^{-n}\int_{S_R^+(\tilde x)}y_n\,u(y)\, d\sigma_R(y) \\
&\ge  c_0|S_1^+|R\,[u]_{\tilde x}(R) \ge 2c_0|S_1^+|L(u)\,x_n.
\end{aligned}$$

\medskip

{\it Step 3.} Define $E=\{c\ge 0;\ u\ge c\,x_n \ \hbox{in $\mathbb{R}^n_+$}\}$.
The set $E$ is closed and nonempty. For any $c\in E$, we have $L(u)\ge c[x_n]$, hence $E$ is bounded and
$$\tilde c:=\max E\le c^*:=[x_n]^{-1}L(u).$$ Assume for contradiction that $\tilde c<c^*$.
Setting $z=u-\tilde c\, x_n$, we see that $z$ is nonnegative, superharmonic and that $L(z)>0$.
By the result of Step~2 applied to $z$, it follows that $z\ge \eps x_n$ for some $\eps>0$.
But this contradicts the definition of $\tilde c$.
Therefore $\tilde c=c^*$ and the result is proved.
\end{proof}

\begin{rmq}
\label{rmqGardiner}
(i) For any subharmonic function $w$ on $\mathbb{R}^n_+$, the
Corollary to Theorem 1 on page 341 in  \cite{Gardiner}
asserts the following: if $w_+$ has a harmonic majorant, if $\displaystyle\liminf_{R\to\infty}\, [w](R) \leq 0$
and if,  for all $y\in \partial \mathbb{R}^n_+$,  $\displaystyle\liminf_{R\to 0} R\, [w]_y(R) \leq 0$, then $w\le 0$.
To deduce Lemma~\ref{lem_lambda}(ii) from this,
set $L=L(u)$ and $w=\frac{L}{[x_n]}\,x_n-u$. Then $w$ is subharmonic,
$w_+$ has a harmonic majorant $\frac{L}{[x_n]}\,x_n$
and $[w](R)=L-[u](R)\underset{R \rightarrow \infty}{\longrightarrow} 0$.
Moreover, for all $y\in \partial \mathbb{R}^n_+$, $\displaystyle\liminf_{R\to 0}R [w]_y(R)
\leq \displaystyle\liminf_{R\to 0}R\textstyle\frac{L}{[x_n]}\,[x_n] =0$.
Therefore, $w\leq 0$, i.e., $u\ge \frac{L}{[x_n]}\,x_n$.

\medskip

(ii) From Lemmas~\ref{lem_derivee_moyenne} and \ref{lem_lambda}(ii),
we may retrieve the well-known fact that any positive harmonic function in $\mathbb{R}^n_+$, such that $u\in C^2(\overline{\mathbb{R}^n_+})$
and $u=0$ on the boundary, is necessarily of the form $u=c\,x_n$ with $c>0$.

We first claim that $L(u)>0$. Indeed, $[u](R)$ is independent of $R$ by Lemma~\ref{lem_derivee_moyenne}.
Therefore $L(u)=0$ would imply $[u](R)\equiv 0$, from which we readily infer $u\equiv 0$.
Let then $z=u-\frac{L(u)}{[x_n]}\,x_n$. Then $z$ is harmonic and Lemma~\ref{lem_lambda}(ii) guarantees $z\ge 0$.
Since $L(z)=0$, the above argument yields $z\equiv 0$.
\end{rmq}

\medskip

\begin{proof}[\underline{Proof of Theorem \ref{thm_systeme_S_rsd_generalise}}.]
(i) Since the functions $f$ and $g$ are nonnegative, $u$ and $v$ are superharmonic. Therefore, by Lemma \ref{lem_derivee_moyenne}(ii),
 \begin{equation}
 \label{notation_lambda_mu}
L(u):=\underset{R\rightarrow \infty}{\lim} [u](R)\in [0,\infty) \qquad  \text{and}\qquad
L(v):= \underset{R\rightarrow \infty}{\lim} [v](R)\in [0,\infty).
  \end{equation}

First, we observe that we cannot have simultaneously $L(u)>0$ and $L(v)>0$.
Indeed, by Lemma~\ref{lem_lambda}(ii), this would imply that, for some $c>0$, and all $x\in \mathbb{R}^n_+$
 $$u(x)\geq c\,x_n \quad \text{ and } \quad v(x)\geq c\,x_n,$$
 hence $-\Delta u\geq (c \,x_n)^{\sigma}$ in $\mathbb{R}^n_+$, but this contradicts
Lemma \ref{lem_estimation_integrale_Liouville}.

Assume for instance $L(u)=0$. Setting $w=u-Kv$, we have $w^+\leq u$, hence
$$\underset{R\rightarrow +\infty}{\lim} [w^+](R)=0.$$
By Lemma \ref{lemme_principal}, this implies $w\leq 0$, i.e. $u\leq Kv$.
If $L(v)=0$, we similarly obtain $u\geq Kv$.

\smallskip

(ii) By what we just proved, it is enough to show that $L(u)=L(v)=0$. Assume $L(v)>0$.
Therefore, $v\ge c x_n$ for $c>0$, and
\begin{equation}
\label{system_ineqs1}
\left\{\quad{\alignedat2
-\Delta u&\geq cx_n^p u^r \\
 -\Delta v&\geq cx_n^s u^q\qquad\mbox{in }\;\mathbb{R}^n_+.
 \endalignedat}\right.
\end{equation}
If the first condition in \eqref{condition_rnsd} is satisfied, then the first inequality in \eqref{system_ineqs1} combined with Lemma \ref{lem_estimation_integrale_Liouville} yields $u\equiv 0$.

Hence we can assume that the second condition in \eqref{condition_rnsd} is satisfied. Set $\Psi :=x_n|x|^{-n}$, so that $-\Delta \Psi = 0 $ in $\mathbb{R}^n_+\setminus\{0\}$.
Let $$c_0:= \inf_{\partial B_1\cap \mathbb{R}^n_+}\frac{u}{\Psi}.$$ Note $c_0>0$ (if $u=0$ on $\partial \mathbb{R}^n_+$, this follows from Hopf's lemma). Since $u$ is superharmonic in $\mathbb{R}^n_+$, $u\ge c\Psi$ on $\partial (\mathbb{R}^n_+\setminus B_1)$ and $\Psi\to 0$ as $|x|\to \infty$, the maximum principle implies
$$u\ge cx_n|x|^{-n}\qquad\mbox{ in }\; \mathbb{R}^n_+\setminus B_1.
$$
Plugging this into the second inequality of \eqref{system_ineqs1} we get
$$ - \Delta v\geq c \, {x_n}^{s+q}|x|^{-nq}$$
in $\mathbb{R}^n_+\setminus B_1$, which contradicts Lemma \ref{lem_estimation_integrale_Liouville},
{applied with $r=0$ and $\kappa=s-(n-1)q$}, 
in view of the second condition in \eqref{condition_rnsd}.

In case $L(u)>0$ we use \eqref{condition_rnsd2} in a similar way, to conclude the proof of (ii).

\smallskip

(iii) By (ii), we know that either $(u,v)$ is semi-trivial or $u=Kv$. In the latter case,
since $\min(p+r, q+s)\leq (n+1)/(n-1)$, we deduce from
Lemma~\ref{lem_estimation_integrale_Liouville} that $u=0$ or $v=0$.
\end{proof}

\mysection{A priori estimates and existence}

We consider the following system with  general lower order terms, of which \eqref{system_pqr_Dir} is a  particular case

\begin{equation}
\label{system_pqr_Dir2}
\left\{\quad{\alignedat2
-\Delta u&=u^rv^p\bigl[a(x)v^q-c(x)u^q\bigr]+h_1(x,u,v), &\qquad&x\in \Omega,\\
 -\Delta v&=v^ru^p\bigl[b(x)u^q-d(x)v^q\bigr]+h_2(x,u,v), &\qquad&x\in \Omega,\\
u&=v=0, &\qquad&x\in \partial\Omega,\\
 \endalignedat}\right.
\end{equation}
where $\Omega$ is a smooth bounded domain of $\mathbb{R}^n$.

Theorem~\ref{AprioriBound} is a consequence of the following more general statements
on a priori estimates and existence.

\begin{thm}
\label{AprioriBound2}
Let $p, r\ge 0$, $q>0$, $q\ge |p-r|$, and
\begin{equation}
\label{hypApriori1}
q+r\ge 1,\qquad 1<\sigma:=p+q+r<\frac{n+2\phantom{_+}}{(n-2)_+}.
\end{equation}
Let $a, b, c, d\in C(\overline\Omega)$ satisfy $a, b>0$, $c, d\ge 0$ in $\overline\Omega$ and
\begin{equation}
\label{hypApriori2}
\inf_{x\in \Omega}\,\left[a(x)b(x)-c(x)d(x)\right]\,>0.
\end{equation}
Let $h_1, h_2\in C(\overline\Omega\times [0,\infty)^2)$ satisfy
\begin{equation}
\label{hypApriori3}
\lim_{u+v\to\infty} \frac{h_i(x,u,v)}{(u+v)^\sigma}=0,\quad i=1,2,
\end{equation}
and assume that one of the following two sets of assumptions is satisfied:
\begin{equation}
\label{hypApriori4}
\left\{\quad{\alignedat2
&r\le 1,\quad \hbox{and, setting }\: \bar m := \min\{\inf_{x\in\Omega} a(x), \inf_{x\in\Omega} b(x)\}>0,\\
&\liminf_{v\to\infty,\ u/v \to 0} \frac{h_1(x,u,v)}{u^rv^{p+q}} >-\bar m,\quad
\liminf_{u\to\infty, \ v/u \to 0} \frac{h_2(x,u,v)}{v^ru^{p+q}} >-\bar m,
\endalignedat}\right.
\end{equation}
or
\begin{equation}
\label{hypApriori5}
\left\{\quad{\alignedat2
&m:=\min\{\inf_{x\in\Omega} c(x), \inf_{x\in\Omega}d(x)\}>0, \quad\mbox{ and } \\
&\limsup_{u\to\infty,\ v/u \to 0} \frac{h_1(x,u,v)}{u^{r+q}v^p} <m,\quad
\limsup_{v\to\infty,\ u/v \to 0} \frac{h_2(x,u,v)}{v^{r+q}u^p} <m
 \endalignedat}\right.
\end{equation}
(with uniform limits with respect to  $x\in\overline\Omega$ in (\ref{hypApriori3})--(\ref{hypApriori5})). 
Then there exists $M>0$ such that any positive classical solution $(u,v)$ of
(\ref{system_pqr_Dir2}) satisfies
\begin{equation}
\label{AprioriBoundM}
\sup_\Omega u \leq M,\qquad \sup_\Omega v \leq M.
\end{equation}
\end{thm}

\begin{thm}
\label{ThmExist}
Let  (\ref{hypApriori1})--(\ref{hypApriori4}) be satisfied.
Assume in addition that $a, b, c, d, h_1, h_2$ are H\"older continuous and that for some $\eps>0$
\begin{equation}
\label{hypExist1}
\inf_{x\in\Omega,\ u, v>0} u^{-1}\,h_1(x,u,v)>-\infty, \qquad
\inf_{x\in\Omega,\ u, v>0} v^{-1}\,h_2(x,u,v)>-\infty,
\end{equation}
\begin{equation}
\label{hypExist2}
\sup_{x\in\Omega,\ u>0} u^{-1}\,h_1(x,u,0)< \lambda_1(-\Delta, \Omega),\qquad
\sup_{x\in\Omega,\ v>0} v^{-1}\,h_2(x,0,v)< \lambda_1(-\Delta, \Omega),
\end{equation}
\begin{equation}
\label{hypExist3}
\sup_{x\in\Omega,\ u,v\in (0, \eps)^2} (u+v)^{-1}\,[h_1(x,u,v)+h_2(x,u,v)]< \lambda_1(-\Delta,\Omega).
\end{equation}
Then there exists  a positive classical solution of (\ref{system_pqr_Dir2}).
\end{thm}

\begin{rmq}
\label{rmqDegree}
We will not treat the existence question under the assumption  (\ref{hypApriori5}),
which seems to be a delicate problem.
 The reason is that we prove Theorem \ref{ThmExist} by using a deformation of the
system (\ref{system_pqr_Dir2}) via homotopy,  adding positive linear terms (see \eqref{system_pqr_Dirhath} below). However,  with such terms,
assumption (\ref{hypApriori5}) is no longer  satisfied and we cannot use Theorem~\ref{AprioriBound2}.
\end{rmq}

\begin{rmq}
\label{remruleout}
Like many previous works, our proof of a priori estimates uses the classical rescaling method of Gidas and Spruck \cite{GS1}.
However, as mentioned in the introduction, arises an additional difficulty: to rule out the possibility of semitrivial rescaling limits, of the form $(C_1,0)$ or $(0,C_2)$ (see Step~2 below).
Under assumption~(\ref{hypApriori4}), this will be achieved by a suitable eigenfunction argument, while (\ref{hypApriori5}) will guarantee that in each blowing up solution $(u,v)$ of (\ref{system_pqr_Dir2}) the components $u$ and $v$ explode at a comparable rate.
Note that a similar difficulty appears in the work~\cite{Zou}, which studied a class of cooperative systems with nonnegative nonlinearities in the form of products.
In that case, the problem was dealt with by different techniques, namely moving planes and Harnack inequalities.
\end{rmq}

For the reader's convenience, before giving the proofs of Theorems~\ref{AprioriBound2}-\ref{ThmExist} we quickly review the role of the hypotheses in these theorems. The first condition in  (\ref{hypApriori1}) guarantees that the strong maximum principle applies to the  system \eqref{system_Spqr}, while the second condition in (\ref{hypApriori1}) is a usual superlinearity and subcriticality condition on the nonlinearities at infinity. The hypothesis (\ref{hypApriori3}) says $h_1$ and $h_2$ are indeed of "lower order", and disappear in the blow-up limit, while the assumptions \eqref{hypApriori4}-\eqref{hypApriori5} are used to exclude semitrivial blow-up limits. The hypothesis \eqref{hypExist1} permits to us to apply the strong maximum principle to \eqref{system_pqr_Dir2}, whereas \eqref{hypExist2} implies that for each nonnegative solution of \eqref{system_pqr_Dir2} we have $u\equiv0$ if and only if $v\equiv0$. Finally, \eqref{hypExist3} is a standard superlinearity condition at zero for \eqref{system_pqr_Dir2}.

\begin{proof}[\underline{Proof of Theorem~\ref{AprioriBound2}}.]
We will consider the following parametrized version of system (\ref{system_pqr_Dir2}) (this will be needed in the proof of Theorem~\ref{ThmExist}):
\begin{equation}
\label{system_pqr_Dir3}
\left\{\ {\alignedat2
-\Delta u&=F(t,x,u,v), &\qquad &x\in \Omega,\\
 -\Delta v&=G(t,x,u,v), &\qquad&x\in \Omega,\\
u&=v=0, &\qquad &x\in \partial\Omega,\\
 \endalignedat}\right.
\end{equation}
where
\begin{equation}
\label{system_pqr_DirF}
F(t,x,u,v):=u^rv^p\bigl[(a(x)+tA)v^q-c(x)u^q\bigr]+\hat h_1(t,x,u,v),
\end{equation}
\begin{equation}
\label{system_pqr_DirG}
G(t,x,u,v):=v^ru^p\bigl[(b(x)+tA)u^q-d(x)v^q\bigr]+\hat h_2(t,x,u,v),
\end{equation}
and
\begin{equation}
\label{system_pqr_Dirhath}
\hat h_1(t,x,u,v)=h_1(x,u,v)+At(1+u),\quad \hat h_2(t,x,u,v)=h_2(x,u,v)+At(1+v).
\end{equation}
Here $A>0$ is a constant to be fixed below, and $t$ is a parameter in $[0,1]$.

Note that (\ref{system_pqr_Dir2}) is (\ref{system_pqr_Dir3}) with $t=0$.
Under assumption (\ref{hypApriori4}), we will prove the bound in (\ref{AprioriBoundM}) for the positive solutions of (\ref{system_pqr_Dir3}), uniformly for $t\in [0,1]$
(but possibly depending on $A$), whereas under assumption (\ref{hypApriori5}) we will restrict ourselves to $t=0$\footnote{The restriction $t_j=0$
under assumption (\ref{hypApriori5}) will be used only in Step~2
to exclude semi-trivial rescaling limits.}
(see~Remark~\ref{rmqDegree}).

We assume for contradiction that there exists a sequence $\{t_j\}\subset [0,1]$ and a sequence $(u_j,v_j)$
of positive solutions of (\ref{system_pqr_Dir3}) with $t=t_j$,
such that $\|u_j\|_\infty+\|v_j\|_\infty \to \infty$.
We may assume $\|u_j\|_\infty \geq \|v_j\|_\infty$ without loss of generality.
Set $\alpha=2/(\sigma-1)$. Let $x_j\in\Omega$ be such that $u_j(x_j)=\|u_j\|_\infty$ and set
$$\lambda_j:=\bigl(\|u_j\| ^{1/\alpha}_\infty+\|v_j\| ^{1/\alpha}_\infty\bigr)^{-1}
\to 0, \quad\hbox{ as $j\to\infty$}.$$
By passing to a subsequence, we may assume that $x_j\to x_\infty\in\overline\Omega$
and $t_j\to t_0\in [0,1]$.
Setting $d_j:=\text{dist}(x_j,\partial\Omega)$, we then split the proof into two cases,
according to whether $d_j/\lambda_j\to\infty$ (along some subsequence) or
$d_j/\lambda_j$ is bounded.
\smallskip

{\bf Case A:} $d_j/\lambda_j\to\infty$.
\smallskip

This case is treated in two steps.
\smallskip

{\it Step 1:} {\it Convergence of rescaled solutions to a semi-trivial entire solution.}

We rescale the solutions around $x_j$ as follows:
\begin{equation}
\label{proofApriori0}
\tilde u_j(y)=\lambda_j^\alpha\, u_j(x_j+\lambda_j y), \quad
\tilde v_j(y)=\lambda_j^\alpha\, v_j(x_j+\lambda_j y),
\qquad y\in\Omega_j,
\end{equation}
where $\Omega_j=\{y\in\mathbb{R}^n: |y|<d_j/\lambda_j\}$.
Due to the definition of $\lambda_j$, it is clear that
\begin{equation}
\label{proofApriori1}
\tilde u_j(y),\ \tilde v_j(y) \leq 1,\qquad y\in\Omega_j.
\end{equation}
Moreover, $\tilde u_j^{1/\alpha}(0)=\lambda_j\,\|u_j\|_\infty^{1/\alpha}
\geq \lambda_j\,(\|u_j\|_\infty^{1/\alpha}+\|v_j\|_\infty^{1/\alpha}\bigr)/2=1/2$, hence
\begin{equation}
\label{proofApriori2}
\tilde u_j(0)\geq 2^{-\alpha}.
\end{equation}
We see that $(\tilde u, \tilde v)=(\tilde u_j, \tilde v_j)$ satisfies the system
\begin{equation}
\label{proofApriori3}
 \left\{\;{\alignedat2
 -\Delta \tilde u &= \tilde u^r \tilde v^p\bigl[(a(x_j+\lambda_j y)+t_j A)\,\tilde v^q-b(x_j+\lambda_j y)\,\tilde u^q\bigr]+\tilde h_{1,j}(y),    &\quad& y\in\Omega_j,\\
 -\Delta \tilde v &= \tilde v^r \tilde u^p\bigl[(b(x_j+\lambda_j y)+t_j A)\,\tilde u^q-d(x_j+\lambda_j y)\tilde v^q\bigr]+\tilde h_{2,j}(y),    &\quad& y\in\Omega_j,\\
  \endalignedat}\quad\right. 
\end{equation}
where
$$\tilde h_{i,j}(y)=\lambda_j^{\alpha+2}\hat h_i(t_j,x_j+\lambda_j y,\lambda_j^{-\alpha}\tilde u_j(y),\lambda_j^{-\alpha}\tilde v_j(y)),
\quad i=1,2.
$$
In view of  (\ref{hypApriori3}), \eqref{proofApriori1}, $\sigma>1$, and $\alpha+2-\alpha\sigma=0$  we have
\begin{equation}
\label{proofApriori4}
\sup_{\Omega_j} \ (|\tilde h_{1,j}|+|\tilde h_{2,j}|) \leq \lambda_j^{\alpha+2}(\lambda_j^{-\alpha \sigma}o(1) + 2A(1+\lambda_j^{-\alpha}) )\to 0,
 \quad\hbox{ as $j\to\infty$.}
\end{equation}
For each fixed $R>0$, we have $B_{2R}\subset\Omega_j$ for $j$ sufficiently large,
and $|\Delta \tilde u_j|,\ |\Delta \tilde v_j|\allowbreak \leq C(R)$ in $B_{2R}$,
owing to (\ref{proofApriori1}), (\ref{proofApriori3}), (\ref{proofApriori4}).
It follows from interior elliptic estimates that
the sequences $\tilde u_j$, $\tilde v_j$ are bounded in $W^{2,m}(B_R)$ for each $1<m<\infty$.
By embedding theorems, we deduce that they are bounded in
$C^{1+\gamma}(\overline{B_R})$ for each $\gamma\in (0,1)$.
It follows that, up to some subsequence,
$$\lim_{j\to\infty}(\tilde u_j, \tilde v_j)=(U,V),\quad\hbox{ locally uniformly on $\mathbb{R}^n$,}$$
where $(U,V)$ is a bounded nonnegative classical solution of
\begin{equation}
\label{proofAprioriUV}
\left\{\;{\alignedat2
 -\Delta U &= U^r V^p\bigl[a_0V^q-c_0U^q\bigr],    &\qquad& y\in\mathbb{R}^n,\\
 -\Delta V &= V^r U^p\bigl[b_0U^q-d_0V^q\bigr],    &\qquad& y\in\mathbb{R}^n,\\
  \endalignedat}\qquad\right. 
\end{equation}
with
\begin{equation}
\label{proofAprioriUV2}
a_0=a(x_\infty)+t_0A>0, \ b_0=b(x_\infty)+t_0A>0,\ c_0=c(x_\infty)\ge 0,\ d_0=d(x_\infty)\ge 0.
\end{equation}
Moreover,
\begin{equation}
\label{proofAprioriUV3}
c_0d_0<a_0b_0
\end{equation}
 in view of (\ref{hypApriori2}).
Also, $U(0)\geq 2^{-\alpha}$ due to (\ref{proofApriori2}).
By Theorem~\ref{thm_SpqrLiouv}(i) and Corollary~\ref{rmqNonneg}, there exists a constant $\bar C>0$ such that $U\equiv \bar C$ and $V\equiv 0$, hence
\begin{equation}
\label{proofApriori5}
\lim_{j\to\infty}(\tilde u_j, \tilde v_j)=(\bar C,0),\quad\hbox{ locally uniformly on $\mathbb{R}^n$.}
\end{equation}

\smallskip

{\it Step 2:} {\it Exclusion of semi-trivial rescaling limits.}

Let us first consider the case when assumption (\ref{hypApriori4}) is satisfied.
For some $\delta, M_1>0$ we have
$$\hat h_2(t,x,u,v)\ge (-\bar m+\delta) v^ru^{p+q},\quad\hbox{ for $u\ge M_1
\max(v,1)$},$$
{(uniformly in $x\in \Omega$ and $t\in [0,1]$)} and hence
$$
\tilde h_{i,j} \ge (-\bar m+\delta) \tilde{v}_j^r \tilde{u}_j^{p+q},
\quad\hbox{ for $\tilde{u}_j\ge M_1  \max(\tilde{v}_j,
{\lambda_j^\alpha),\quad i=1,2.}$}$$
Fix $\eps\in (0,1)$ with
$$\eps\le \min\Bigl\{\frac{\bar C}{2M_1},
 \Bigl(\frac{\delta}{{2}\|d\|_\infty}\Bigr)^{1/q}\frac{\bar C}{2}\Bigr\}.$$
Take $R>0$ to be chosen later. By (\ref{proofApriori5}), there exists $j_0$ such that, for all $j\ge j_0$, we have
$\tilde u_j\ge \bar C/2$, $\tilde v_j\le \eps$ on $B_R$, and
$\tilde{u}_j\ge \bar C/2 \ge M_1  \max(\tilde{v}_j, \lambda_j^\alpha)$, since $\lambda_j^\alpha \to 0$ as $j\to\infty$. Hence
\begin{eqnarray*}
 -\Delta \tilde v_j &\ge& \tilde v_j^r \tilde u_j^p\bigl[\bigl(b(x_j+
\lambda_j y)+t_jA-\bar m+\delta\bigr)\,\tilde u_j^q-d(x_j+\lambda_j y)
\tilde v_j^q\bigr]\\
       &\ge& \tilde v_j^r \tilde u_j^p\bigl[\textstyle {\delta}\,
\tilde u_j^q-\|d\|_\infty\eps^q]
       \ge   \textstyle\frac{\delta}{2}\bigl(\frac{\bar C}{2}\bigr)^{p+q}
\tilde v_j^r
       \ge {\textstyle\frac{\delta}{2}}\bigl(\frac{\bar C}{2}\bigr)^{p+q}
\tilde v_j\qquad \mbox{in }\; B_R,
 \end{eqnarray*}
(in the last inequality we used $r\le 1$). This implies that the first
eigenvalue of the Laplacian in $B_R$ is larger than
 $\frac{\delta}{{2}}\bigl(\frac{\bar C}{2}\bigr)^{p+q}$, which is a contradiction for sufficiently large $R$. More precisely, denote by $\lambda_1(R)$ and $\varphi_R$ the first eigenvalue and eigenfunction of $-\Delta$ in $B_R$
with Dirichlet boundary conditions. Since $\lambda_1(R)=\lambda_1(1)R^{-2}$, by multiplying the above inequality with $\varphi_R$
and by integrating by parts, we get
$$\lambda_1(1)R^{-2}\int_{B_R}\tilde v_j  \varphi_R\, dx =
-\int_{B_R}\tilde v_j  \Delta\varphi_R\, dx  \ge
-\int_{B_R} \varphi_R\Delta\tilde v_j \, dx  \ge \textstyle\frac{\delta}{4}\bigl(\frac{\bar C}{2}\bigr)^{p+q}
\displaystyle\int_{B_R}\tilde v_j  \varphi_R\, dx.$$
By taking $R$ sufficiently large (depending only on $\delta, \bar C, p, q$), this implies $\tilde v_j =0$ on $B_R$, a contradiction.

Let us now consider the case when assumption (\ref{hypApriori5}) is satisfied, and $t_j=0$.
Now there exist $\delta, M_1>0$ such that
$$h_1(x,u,v)\le (m-\delta) u^{r+q}v^p, \quad\hbox{ if }\; u\ge M_1 \max(v,1).$$
Therefore, for any positive solution $(u,v)$ of (\ref{system_pqr_Dir2}), if $\|u\|_\infty\ge M_1$ then, at a maximum point $x_0$ of $u$, we have
either $u(x_0)<M_1 v(x_0)$, or else
$$0\le -\Delta u(x_0)\le u^rv^p[av^q-(c-m+\delta)u^q](x_0).$$
Since $u$ and $v$ are positive  we deduce that
$${ v(x_0)}\ge \Bigl(\frac{\delta}{a(x_0)}\Bigr)^{1/q}u(x_0)\ge
\Bigl(\frac{\delta}{\|a\|_\infty}\Bigr)^{1/q}{ u(x_0)}.$$
Hence there exists a constant $\eta>0$
such that, for any positive solution $(u,v)$ of (\ref{system_pqr_Dir2}),
$$\|u\|_\infty{= u(x_0)}\ge {M_1} \Longrightarrow {v(x_0)}\ge \eta\,{u(x_0)}.$$
In view of definition (\ref{proofApriori0}), this implies $\tilde v_j(0)\ge \eta\tilde u_j(0)$,
hence $V(0)\ge\eta U(0)\ge\eta2^{-\alpha}$, which excludes semitrivial limits and leads  to a contradiction with the nonexistence of positive solutions of \eqref{proofAprioriUV}.

\smallskip
{\bf Case B:} $d_j/\lambda_j$ is bounded.
We may assume that $d_j/\lambda_j\to c_0\geq0$.
Arguing similarly to~\cite[pp.~891-892]{GS1}
 (see also~\cite[p.~265]{QS}), after performing local changes of coordinates
which flatten the boundary, we end up with a nontrivial
nonnegative (bounded) solution $(U,V)$ of system (\ref{proofAprioriUV}) in a half-space,
with $U=V=0$ on the boundary.
Moreover, (\ref{proofAprioriUV3}) is satisfied. 
By Proposition~\ref{Prop-uKv} and Theorem \ref{thm_general}, we deduce $U=KV$, $K>0$, which in turn implies that $-\Delta U=C_1U^\sigma$, $-\Delta V=C_2V^\sigma$ in the half-space, for some $C_1, C_2>0$.
This yields a contradiction with the Liouville-type theorem in \cite{GS1} for half-spaces.
\end{proof}

\begin{proof}[\underline{Proof of Theorem~\ref{ThmExist}}.]
First, it is important to observe that the assumptions $q+r\ge 1$ and (\ref{hypExist1}) guarantee that any
nonnegative solution of (\ref{system_pqr_Dir3})  satisfies $u>0$ and $v>0$ in~$\Omega$,
unless $t=0$ and $(u,v)\equiv (0,0)$.
Indeed, if $u\not\equiv 0$, then $u>0$  by the strong maximum principle -- note by \eqref{system_pqr_DirF} and (\ref{hypExist1}) we have $$F(t,x,u,v)\ge -C u,$$ for some $C\ge 0$ (which may depend on $t,u,v,c,d,A$).
On the other hand, assume for instance $u\equiv 0$. Then
$0=F(t,x,0,v) \ge At$ (since (\ref{hypExist1}) implies $h_1\ge -C_1u$ for some $C_1\ge0$), so  $t=0$. Then (\ref{hypExist2})  
implies  that $$-\Delta v\le h_2(x,0,v)\le \bigl(\lambda_1(-\Delta, \Omega)-\varepsilon_0\bigr) v\quad\mbox{ in }\;\Omega,$$ for some $\varepsilon_0>0$.
We then easily deduce $v\equiv0$, by multiplying with the first Dirichlet eigenfunction of $-\Delta$
and by integrating by parts.

Theorem~\ref{ThmExist} follows from a standard topological degree argument. We recall the following fixed point theorem, due to Krasnoselskii and Benjamin (see Proposition 2.1 and Remark 2.1 in  \cite{dFLN}). This type of statements are nowadays standard in proving existence results.
\begin{thm}\label{theofixed}  Let ${\cal K}$ be a
closed cone in a Banach space $E$, and let  $T\,:\, {\cal K} \rightarrow {\cal K}$ be a
compact mapping. Suppose $0<\delta<M<\infty$,
are such that
\begin{description}
\item[(i)] $\eta Tx\not=x$ for all $x\in {\cal K}, \|x\|=\delta$,  and all
$\eta\in[0,1]$;
\end{description}
and there exists a compact mapping $H\,:\, {\cal K}\times
[0,1]\rightarrow {\cal K}$ such that
\begin{description}
\item[(ii)] $H(x, 0) =Tx$ for all $x\in {\cal K}$;
\item[(iii)] $H(x,t) \not = x$ for all $x\in {\cal K}, \|x\|=M$ and
all $t\in [0,1]$;
\item[(iv)] $H(x,1)\not =x $ for all  $x\in {\cal K}, \|x\|\leq M$. \end{description} Then there exists a fixed point $x$ of $T$
(i.e. $Tx=x$), such that $\delta\le\|x\|\le M$.
\end{thm}
Observe that (i) implies $i(T,B_\delta\cap {\cal K}) = i(0,B_\delta\cap {\cal K}) =1$, where $i$ is the (homotopy invariant) fixed point index with respect to the relative topology of ${\cal K}$, whereas by (iii)-(iv)
$$
i(H(\cdot,0),B_R\cap {\cal K}) = i(H(\cdot,1),B_R\cap {\cal K}) =0,
$$
and the excision property of the index implies Theorem \ref{theofixed}.\smallskip

A little care is needed in defining $T$ and $H$.
Let ${\cal K}$ denote the  cone of nonnegative functions in
$E:=C(\overline\Omega)\times C(\overline\Omega)$,
and let  ${\cal T}:E\to {\cal K}$ be defined by $${\cal T}(\phi,\psi)=(u_+,v_+),$$
where $(u,v)$ is the solution of the linear problem
$$ \begin{aligned}
   -\Delta u &=\phi, \ -\Delta v=\psi \ \hbox{in }\Omega,\\
   u &=v=0 \ \hbox{ on }\partial\Omega.
\end{aligned}
$$
{It is clear that ${\cal T}$ is compact, since $(\phi,\psi)\to (u,v)$ is such by elliptic estimates, and $(u,v)\to (u_+,v_+)$ is Lipschitz.}
We set
$$
H((u,v),t):= {\cal T}\bigl(F(t,x,u(x),v(x)),G(t,x,u(x),v(x))\bigr),
$$
and $T(u,v) = H((u,v),0)$. Recall $F,G$ are defined in \eqref{system_pqr_DirF}-\eqref{system_pqr_DirG}, so fixed points of $H(\cdot,t)$ are solutions of (\ref{system_pqr_Dir3}).

We still have to choose the constant $A$ in
 \eqref{system_pqr_DirF}--\eqref{system_pqr_Dirhath}. 
We do this in the following way: by (\ref{hypExist1}), there exists $C_1>0$ such that $h_1\ge -C_1u$ and $h_2\ge -C_1v$, and we set
\begin{equation}
\label{choiceA}
A= \max\left\{C_1+\lambda_1(-\Delta,\omega),\ \sup_{x\in\Omega}c(x),\ \sup_{x\in\Omega}d(x)\right\},
\end{equation}
where $\omega$ is some smooth strict subdomain of $\Omega$.
{Once $A$ is fixed, we know from the proof of Theorem \ref{AprioriBound2} that there
exists a universal bound for the positive solutions (if they exist) of
(\ref{system_pqr_Dir3})
valid for all $t\in[0,1]$,  and we chose $M$ larger than this bound.}

Theorem~\ref{ThmExist} is proved if we show that $T$ has a nontrivial fixed point in ${\cal K}$. So it remains to check that the hypotheses of Theorem~\ref{theofixed} are satisfied.

$\bullet$ Let us first show that $H(\cdot,1)$ does not possess any fixed point in ${\cal K}$,
which will verify~(iv).
Assume such  a fixed point $(u,v)$ exists, which is then a solution of (\ref{system_pqr_Dir3}), with $t=1$. We have $u, v>0$ in $\Omega$, since $t>0$.
Let ${\cal S}=u^{1/2}v^{1/2}$. By using the inequality $2\Delta ((uv)^{1/2}) {\le} v^{1/2}u^{-1/2}\Delta u+u^{1/2}v^{-1/2}\Delta v$, we get 
\begin{align*}
-\Delta {\cal S}
&\geq\frac{u^{-1/2}v^{1/2}}{2}\Bigl[u^r v^p\bigl((a(x)+A)v^q-c(x) u^q\bigr)+(A-C_1)u+A\Bigr]\\
\noalign{\vskip 1mm}
&\qquad \qquad +\frac{u^{1/2}v^{-1/2}}{2}\Bigl[v^r u^p\bigl((b(x)+A)u^q- d(x) v^q\bigr)+(A-C_1)v+A\Bigr]\\
\noalign{\vskip 1mm}
&\geq \frac{v^\sigma X^{-1/2}}{2}\bigl[(a(x)+A)X^r +(b(x)+A)X^{p+q+1}-c(x)X^{q+r}-d(x) X^{p+1}\bigr]\\
\noalign{\vskip 1mm}
&\qquad \qquad +(A-C_1){\cal S} +A,
\end{align*}
where $X=u/v$.
Using (\ref{choiceA}) and the inequality
$$X^r +X^{p+q+1}-X^{q+r}-X^{p+1}=X^r(1-X^{q})(1-X^{p+1-r}) \ge 0$$
(note $p+1\ge1\ge r$),
it follows that
$$-\Delta {\cal S}\geq (A-C_1){\cal S} \quad \mbox{ in }\;\omega.$$
We reach a contradiction by testing this inequality with the first Dirichlet eigenfunction in $\omega$, because of \eqref{choiceA}.

$\bullet$ Hypothesis (iii) in Theorem \ref{theofixed} is a consequence of the a priori bound for positive solutions of (\ref{system_pqr_Dir3}) which we obtained in the proof of Theorem~\ref{AprioriBound2}, and  the observation we made in the beginning of the proof of Theorem~\ref{ThmExist}.

$\bullet$ Finally, assume that hypothesis (i) is not verified,
which {implies} that for any (small) $\delta>0$ we can find a
 positive solution $(u,v)$ with $\|(u,v)\|\le \delta$,
of (\ref{system_pqr_Dir2}) with the right-hand side of this system 
multiplied by {some} $\eta\in[0,1]$. By adding up the two equations in the system and using \eqref{hypExist3} we obtain, 
 with $\lambda_1=\lambda_1(-\Delta, \Omega)$ and for some $\eps_0>0$,
\begin{eqnarray*}
-\Delta(u+v) &\leq & C (u^rv^{p+q} + v^ru^{p+q}) + (\lambda_1 -\eps_0) (u+v)\\
&\leq& C(u+v)^{\sigma-1} (u+v) + (\lambda_1 -\eps_0) (u+v)\\
&\leq& (\lambda_1 -\eps_0/2) (u+v)
\end{eqnarray*}
(we obtained the last inequality by choosing $\delta$ sufficiently small). By testing again with the first Dirichlet eigenfunction we get a contradiction.

Theorem~\ref{ThmExist} is proved.
\end{proof}

\begin{rmq}
\label{rmqGeneralization}
By simple modifications of the above proof, one can show that assumption (\ref{hypExist1})
can be weakened as follows: for each $R>0$,
$$\inf_{x\in\Omega\atop u,v\in (0,R)} u^{-1}\,h_1(x,u,v)> -\infty,
\quad \inf_{x\in\Omega\atop u,v\in (0,R)} v^{-1}\,h_2(x,u,v) > -\infty
$$
(which allows for the application of the strong maximum principle)
and, for each $\eps>0$, there exists $C_\eps>0$ such that, for all $u,v\ge 0$, $x\in\Omega$,
$$h_1(x,u,v)\ge -\eps u^rv^{p+q}-C_\eps(1+u),
\quad
h_2(x,u,v)\ge -\eps v^ru^{p+q}-C_\eps(1+v).
$$
\end{rmq}
\medskip

\mysection{Appendix. Proof of Proposition~\ref{Prop-uKv}}

In the appendix we study the proportionality constants of the
system (\ref{system_Spqr}) and give the elementary proof of Proposition~\ref{Prop-uKv}. In the following we set
$$
f(u,v) = u^rv^p[av^q-cu^q], \qquad g(u,v) = v^ru^p[bu^q-dv^q].
$$

\begin{proof}[\underline{Proof of Proposition~\ref{Prop-uKv}}.]
We first note that if $c=d=0$ and $q=r-p$, then $Kg-f=(Kb-a)u^rv^r$
and~\eqref{condition_fg} is satisfied if and only if $K=a/b$
(and then actually $Kg-f\equiv 0$).
We may thus assume that either
\begin{equation}\label{c_or_d_pos}
c>0,\ \ \hbox{or } d>0,\ \ \hbox{or } q\ne r-p.
\end{equation}

\smallskip
Set $X=u/v$. For given $K>0$, we compute, for $u, v>0$,
$$
 \begin{aligned}
Kg-f&=Ku^pv^r(bu^q-dv^q)-u^rv^p(av^q-cu^q)\\
 &=u^pv^{r+q}[KbX^{q}-Kd-aX^{r-p}+cX^{q+r-p}]
 \end{aligned}$$
 and also, factorizing by $X^{r-p}$,
 $$Kg-f=u^rv^{p+q}[KbX^{q+p-r}-KdX^{p-r}-a+cX^{q}].$$
Set
$$m:=|r-p|\le q,$$
 and define
$$H_K(X)=AX^{q+m}+BX^{q}-CX^{m}-D,\qquad X>0,$$
where
$$\begin{cases}
A=c, \ B=Kb, \ C=a, \ D=Kd, &\quad\hbox{ if $r\ge p$, }\\
\noalign{\vskip 1mm}
A=Kb, \ B=c, \ C=Kd, \ D=a, & \quad\hbox{ otherwise. }\\
\end{cases}
$$
We then see that we may write
\begin{equation}
\label{Kg-f}
Kg-f=
\begin{cases}
u^pv^{r+q}H_K(u/v), &\quad\hbox{ if $r\ge p$, }\\
\noalign{\vskip 1mm}
u^rv^{p+q}H_K(u/v), & \quad\hbox{ otherwise. }\\
\end{cases}
\end{equation}

We next claim that there exists (at least one) $K>0$ such that $H_K(K)=0$.
Indeed, setting
$$J(K):=H_K(K)=
\begin{cases}
cK^{q+m}+bK^{q+1}-aK^{m}-dK, &\quad\hbox{ if $r\ge p$, }\\
\noalign{\vskip 1mm}
bK^{q+m+1}+cK^{q}-dK^{m+1}-a, & \quad\hbox{ otherwise, }\\
\end{cases}
$$
 we easily see that
$\lim_{t\to\infty} J(t)=\infty$ and $J(K)<0$ for small $K>0$,
and the claim follows.

Now pick any $K>0$ such that $H_K(K)=0$. We will prove that
\begin{equation}\label{Kg-fStrict}
[Kg(u,v)-f(u,v)][u-Kv]>0  \ \hbox{ for all $u,v>0$ with $u\neq Kv$,}
\end{equation}
a slightly stronger property than \eqref{condition_fg},
which will in particular establish the existence part of Proposition~\ref{Prop-uKv}.

We first consider the case $m>0$ and set $\ell=q/m\ge 1$.
Let us rewrite
\begin{equation}\label{HKX}
H_K(X)= h_K(X^m),
\qquad\hbox{with }h_K(t)=At^{\ell+1}+Bt^{\ell}-Ct-D, \quad t>0.
\end{equation}
This function is easier to handle than $H_K$ because its last two terms are affine
and $h_K$ is convex.
We claim that
\begin{equation}\label{ClaimhK}
\hbox{$h_K(t)<0$ for $t>0$ small.}
\end{equation}

$\bullet$ If $D>0$, then  $h_K(0)=-D<0$.

$\bullet$ If $\ell>1$ and $D=0$, then $h_K(0)=0$ and $h_K'(0)=-C=-a<0$.

$\bullet$ If $\ell=1$ (hence $q=m$) and $D=0$ (hence $r\ge p$ and $d=0$), then we may assume $c>0$ (see~(\ref{c_or_d_pos})).
We have $h_K(s)=cs^2+(Kb-a)s$.
Then necessarily $Kb-a<0$ (since $H_K(K)=0$).
Thus $h_K(0)=0$ and $h_K'(0)<0$.

\par\noindent In either case, (\ref{ClaimhK}) is true. On the other hand, we also have
\begin{equation}\label{ClaimhK2}
\lim_{t\to\infty} h_K(t)=\infty.
\end{equation}
(This is clear unless $A=0$ and $\ell=1$, but in that case $D>0$ due to (\ref{c_or_d_pos}),
hence $B>C$ due to $H_K(K)=0$.)
Now, since $h_K$ is convex on $[0,\infty)$,
it follows from (\ref{ClaimhK}), (\ref{ClaimhK2}) that $h_K$ has a unique zero on $(0,\infty)$.
Consequently, $K$ is the unique zero of $H_K$ on $(0,\infty)$
and, by (\ref{Kg-f}), (\ref{HKX}) and (\ref{ClaimhK}), we deduce (\ref{Kg-fStrict}).

If $m=0$, then $H_K(X)=(Kb+c)X^q-(a+Kd)$, which is monotonically
increasing in $X$, and (\ref{Kg-fStrict}) is clear.
The proof of the existence part is thus complete.

\smallskip

Let us now suppose $ab\ge cd$ and show the uniqueness of $K$.
Assume for contradiction that \eqref{condition_fg} is true for
two distinct values of $K$, say $K_2>K_1>0$.
Pick $Y\in (K_1, K_2)$. For $i\in\{1,2\}$, since $H_{K_i}(K_i)=0$ due to (\ref{Kg-f}), it follows
from what we already proved that (\ref{Kg-fStrict}) is true for $K=K_i$.
In particular, $K_1g(Y,1)>f(Y,1)>K_2g(Y,1)$.
Therefore $g(Y,1)<0$ and $f(Y,1)<0$, that is
$0<a<cY^q$ and $0<bY^q<d$.
Consequently $ab<cd$: a contradiction.

\smallskip

Finally, suppose $ab>cd$ and assume for contradiction that $cK^q\ge a$ (hence $c>0$).
Then $bK^{1+q+p-r} \ge (ab/c)K^{1+p-r} > dK^{1+p-r}$.
It then follows from (\ref{Kg-fStrict})  that $$0=Kg(K,1)-f(K,1)\ge K^r[bK^{1+q+p-r}-dK^{1+p-r}-a+cK^{q}]>0.$$
This contradiction shows  that $a-cK^q>0$.
 The proofs of $bK^q-d>0$ and of the equalities  are similar.
\end{proof}
\medskip

In the end, we prove {related} lower bounds  which we use in the proofs of Lemma \ref{DeltaW}
and  Proposition~\ref{propo_cdzero}.

\begin{lem}\label{elemlem}  Assume (\ref{Hyp_abcdpqr}).
\smallskip

(i) Assume $r>p$ and $c, d>0$.
Then the nonlinearities in the system \eqref{system_Spqr} satisfy, for some $C>0$,
$$
(Kg-f)(u-Kv)\geq C u^p v^p(u+Kv)^{q+r-p-1} (u-Kv)^2.
$$

\smallskip
(ii) Assume $d=0$ and $c>0$. Then
\begin{equation}
\label{eq23}
(Kg-f)(u-Kv)\geq
C u^r v^{p\wedge r}(u+Kv)^{q-1+(p-r)_+} (u-Kv)^2.
\end{equation}

\end{lem}

\begin{proof}[\underline{Proof}]
(i) We use the same  notation as in the above proof of
Proposition~\ref{Prop-uKv}.
First note that ${h_K}'(K^m)>0$ since $h_K$ is negative and convex on $(0,K^m)$.
Denoting $$p(t)=\dfrac{h_K(t)}{t^{\ell +1}-K^{m(\ell +1)}},$$
we observe that $p(t)>0$ on $[0,K^m)\cup(K^m,+\infty)$ and that $p(t)$ has
positive limits as $t$ goes to $K^m$ or $+\infty$ (using that ${h_K}'(K^m)>0$).
Hence, for all $t\geq 0$, we have $p(t)\geq {C}$ for some constant $C>0$. 
So,
$$\dfrac{H_K(X)}{X^{m(\ell +1)}-K^{m(\ell +1)}}\geq C$$ and we obtain
$$H_K(X)(X-K)\geq C (X-K)(X^{m(\ell +1)}-K^{m(\ell +1)}).$$
Since
$(Kg-f)(u-Kv)=u^p v^{r+q+1}H_K(X)(X-K)$,
by using the inequality
$$(x^k-y^k)(x-y)\ge C_k(x+y)^{k-1}(x-y)^2,\quad x,y>0$$
for $k>0$ and some $C_k>0$, we get
$$
\begin{aligned}
(Kg-f)(u-Kv)
&\geq C \dfrac{u^p v^{r+q+1}}{v^{1+q+r-p}} (u^{q+r-p}-(Kv)^{q+r-p}) (u-Kv) \\
&\geq C u^p v^p(u+Kv)^{q+r-p-1} (u-Kv)^2.
\end{aligned}
$$

{(ii) Letting $X=u/v$, we have, for $u,v>0$,
 $$Kg-f=u^rv^{p+q}G(X),\quad\hbox{where }
G(X)=KbX^{q+p-r}+cX^{q}-a.$$
We know from Proposition~\ref{Prop-uKv} that $G$ vanishes only at $X=K$.
Since $G'(K)>0$, it is easy to see that
$$\frac{G(X)}{X-K}\ge C(X+1)^\ell,\quad X\in [0,\infty)\setminus\{K\},$$
where $\ell=q-1+(p-r)_+$.
Therefore
$$(Kg-f)(u-Kv) = u^rv^{p+q-1} \frac{G(X)}{X-K} (u-Kv)^2
 \ge Cu^rv^{p+q-1-\ell}(u+v)^\ell (u-Kv)^2.$$
The assertion follows.}
\end{proof}


{\footnotesize

\begin{thebibliography}{00}
\bibitem{AF}
Alves C.O., de Figueiredo D.G.,
Nonvariational elliptic systems,
Discrete Contin. Dynam. Systems 8 (2002), 289-302.

\bibitem{AD}
Ambrosio L., Dancer N.,
Calculus of  variations and partial differential equations,
Springer, Berlin 1999.

\bibitem{armsir} Armstrong S.N.; Sirakov B.,
 Nonexistence of positive supersolutions of elliptic equations via the maximum principle. Comm. Partial Differential Equations 36 (2011), 2011-2047.

\bibitem{BDW} Bartsch Th., Dancer N., Wang Z.-Q.,
A Liouville theorem, a-priori bounds, and bifurcating branches of positive solutions for a nonlinear elliptic system,
Calc. Var. Partial Differ. Equations
37 (2010), 345-361.

\bibitem{BPB}
Belmonte-Beitia J., P\'erez-Garc\'\i a V., Brazhnyi V.,
Solitary waves in coupled nonlinear Schr\"odinger equations with spatially inhomogeneous nonlinearities,
Commun. Nonlinear Sci. Numer. Simul. 16 (2011), 158-172.

\bibitem{BVGY} Bidaut-V\'eron M.-F., Garcia-Huidobro M., Yarur C.,
Keller-Osserman estimates for some quasilinear elliptic systems,
Commun. Pure Appl. Anal., to appear.

\bibitem{Br} Brezis H., Semilinear equations in $\mathbb{R}^N$ without conditions at infinity,
Applied Math. and Optimization 12 (1984), 271-282.

\bibitem{brecab} Brezis H.; Cabré X.,
Some simple nonlinear PDE's without solutions.
Boll. Unione Mat. Ital.  (8) 1 (1998), 223-262.

\bibitem{CGS} Caffarelli L.A., Gidas B., Spruck J.,
Asymptotic symmetry and local behavior of semilinear elliptic equations with critical Sobolev growth,
Comm. Pure Appl. Math. 42 (1989), 271-297.

 \bibitem{CLZ} Chen Z., Lin C.-S., Zou W.,
 Monotonicity and nonexistence results to cooperative systems in the half space,
 J. Funct. Anal. 266 (2014), 1088-1105.

\bibitem{Che}
Cheng Y., Effective potential of two coupled binary matter wave bright solitons with spatially modulated nonlinearity,
J. Phys. B: At. Mol. Opt. Phys. 42, 205005 (2009).

\bibitem{CMM} Cl\'ement Ph., Man\'asevich, R., Mitidieri E.,
Positive solutions for a quasilinear system via blow up,
Comm. Partial Differential Equations 18 (1993),  2071-2106.

\bibitem{Damb} D'Ambrosio L.,
A new critical curve for a class of quasilinear elliptic systems,
Nonlinear Analysis 78 (2013), 62-78.

\bibitem{Dan} Dancer N., Some notes on the method of moving planes,
Bull. Austral. Math. Soc. 46 (1992), 425-434.

\bibitem{DWW} Dancer N., Wei J.-C., Weth T.,
A priori bounds versus multiple existence of positive solutions for
a nonlinear Schr\"odinger system, Ann. Inst. H. Poincar\'e Anal. Non Lin\'eaire 27 (2010), 953-969.

\bibitem{DancerWeth} Dancer N., Weth T.,
Liouville-type results for non-cooperative elliptic systems in a half-space,
J. London Math. Soc. 86 (2012), 111-128.

\bibitem{DLS}
Delgado M., L\'opez-G\'omez J., Su\'arez A.,
On the symbiotic Lotka-Volterra model with diffusion and transport effects,
J. Differential Equations 160 (2000), 175-262.

\bibitem{DF}
Desvillettes L., Fellner K.,
Exponential Decay toward Equilibrium via Entropy Methods for Reaction-Diffusion Equations,
J. Math. Anal. Appl. 319 (2006), 157-176.

\bibitem{dFLN} de Figueiredo D., Lions P.L., Nussbaum R.,
A priori estimates and existence of positive solutions of
semilinear elliptic equations, J. Math. Pures Appl. {61} (1982),
41-63.

\bibitem{FigSir} de Figueiredo D.G., Sirakov B.,
 Liouville type theorems,
monotonicity results and a priori bounds for positive solutions of elliptic systems. Math. Ann. 333 (2005), 231-260.

\bibitem{FigYang} de Figueiredo D.G., Yang J.,
 A priori bounds for positive solutions of a non-variational elliptic system,
Comm. Partial Differential Equations 26 (2001), 2305-2321.

\bibitem{Fa1} Farina A.,
On the classification of solutions of the Lane-Emden equation on unbounded domains of $\mathbb{R}^N$,
J. Math. Pures Appl. 87 (2007) 537-56.

\bibitem{Fa} Farina A., Symmetry of components and Liouville-type theorems for nonlinear elliptic systems,
preprint 2013.

\bibitem{Gardiner} Gardiner S.J.,
 Half-spherical means and boundary behaviour of subharmonic functions in half-spaces,
 Hiroshima Math. J. 13 (1983), 339-348.

\bibitem{GS1}  Gidas B., Spruck J.,
 A priori bounds for positive solutions of nonlinear elliptic equations,
 Comm. Partial Differential Equations 6 (1981), 883-901.

\bibitem{GS2} Gidas B.,  Spruck J.,
 Global and local behavior of positive solutions of nonlinear elliptic equations,
 Comm. Pure Appl. Math. 34 (1981), 525-598.


\bibitem{GuoLiu} Guo Y.-X., Liu J.-Q.,
Liouville type theorems for positive solutions of elliptic system in $\mathbb{R}^N$,
Comm. Partial Differ. Equations 33 (2008), 263-284.

\bibitem{KoLe} Korman P., Leung A.,
On the existence and uniqueness of positive steady states in the Volterra-Lotka ecological models with diffusion,
 Appl. Anal. 26 (1987), 145-160.

 \bibitem{LiMa}  Li C., Ma L.,
Uniqueness of positive bound states to Schr\"odinger systems with critical exponents,
SIAM J. Math. Anal. 40 (2008), 1049-1057.

\bibitem{Lin} Lin T.-C.,  Wei J.-C.,
 Symbiotic bright solitary wave solutions of coupled nonlinear
Schr\"odinger equations, Nonlinearity 19 (2006), 2755-2773.

\bibitem{Lin2} Lin F.-H. On the elliptic equation $D\sb i[a\sb {ij}(x)D\sb jU]-k(x)U+K(x)U\sp p=0$. Proc. Amer. Math. Soc. 95 (1985), 219-226.

\bibitem{Lou} Lou Y.,
Necessary and sufficient condition for the existence of positive solutions of certain cooperative system,
Nonlinear Anal. 26 (1996), 1079-1095.

\bibitem{MiPo} Mitidieri E.,  Pohozaev S. I.,
 A priori estimates and the absence of solutions of nonlinear partial differential equations and inequalities (Russian),
  Tr. Mat. Inst. Steklova 234 (2001), 1-384; translation in Proc. Steklov Inst. Math. 234 (2001), 1-362.

\bibitem{NiWM} W.-M. Ni, On the elliptic equation $\Delta U + KU^{(n+2)/(n-2)}= 0$, its generalization and application in
geometry, Indiana Univ. Math. J. 31 (1982), 493-529.

\bibitem{Osserman} Osserman, R.,
On the inequality $\Delta u\geq f(u)$. Pacific J. Math. 7 (1957), 1641-1647.

\bibitem{Perez-Garcia}  Perez-Garcia V. M.,  Beitia J. B.,
 Symbiotic solitons in heteronuclear multicomponent
Bose-Einstein condensates, Phys. Rev. A, 72 (2005), 033620.

\bibitem{Pierre} Pierre M.,
Global existence in reaction-diffusion systems with control of mass: a survey,
Milan J. Math. 78 (2010),  417-455.

\bibitem{prowei} Protter M. H., Weinberger H.,  Maximum principles in differential equations.
Corrected reprint of the 1967 original. Springer-Verlag, New York, 1984.

 \bibitem{QS} Quittner P., Souplet Ph., Symmetry of components for semilinear elliptic systems, SIAM J. Math. Analysis 44 (2012), 2545-2559.

 \bibitem{QScmp} Quittner P., Souplet Ph.,
 Optimal Liouville-type theorems for noncooperative elliptic Schr\"odinger systems and applications,
Comm. Math. Phys. 311 (2012), 1-19.

 \bibitem{Reichel-Zou} Reichel W., Zou H.,
 Non-existence results for semilinear cooperative elliptic systems via moving spheres,
 J. Differential Equations 161 (2000), 219-243.

\bibitem{SerrinZou} Serrin J., Zou H.,
Non-existence of positive solutions of Lane-Emden systems,
Differ. Integral Equations 9 (1996), 635-653.

\bibitem{Sir} Sirakov, B., On the existence of solutions of Hamiltonian elliptic systems in RN. Adv. Differential Equations 5 (2000), 1445-1464.

\bibitem{SoupletAdv} Souplet Ph., The proof of the Lane-Emden conjecture in four space dimensions,
Adv. Math. 221 (2009), 1409-1427.

\bibitem{SoupletNHM}  Souplet Ph., Liouville-type theorems for elliptic Schr\"odinger systems associated with copositive matrices,
Netw. Heterog. Media 7 (2012), 967-988.

\bibitem{TTVW}  Tavares H., Terracini S., Verzini G.-M., Weth T.,
Existence and nonexistence of entire solutions for non-cooperative cubic elliptic systems,
Comm. Partial Differ. Eq. 36 (2011), 1988-2010.

 \bibitem{Zou} Zou H.,
 A priori estimates and existence for strongly coupled semilinear cooperative elliptic systems.
 Comm. Partial Differential Equations 31 (2006), 735-773.


\end{thebibliography}
}
\end{document}